\def\isarxiv{1}
\ifdefined\isarxiv
\documentclass[11pt]{article}
\else
\documentclass[a4paper,UKenglish,cleveref, autoref, thm-restate]{lipics-v2021}
\fi
\ifdefined\isarxiv
\else
\fi
\usepackage{tabu}
\usepackage{array}
\usepackage{amsmath}
\usepackage{amsthm}
\usepackage{amssymb}
\usepackage[linesnumbered,ruled,procnumbered]{algorithm2e}
\usepackage{dirtytalk}
\usepackage{color}
\usepackage{url}
\usepackage{color}
\usepackage{epstopdf}
\usepackage{algpseudocode}
\usepackage[T1]{fontenc}
\usepackage{bbm}
\usepackage{comment}
\usepackage{dsfont}
\usepackage{mathtools}

\let\C\relax
\usepackage{tikz}
\usepackage{hyperref}  
\ifdefined\isarxiv
\hypersetup{
    colorlinks=true,
    citecolor=red,
    linkcolor=blue,
    filecolor=cyan,
    urlcolor=magenta,
}
\fi
\usetikzlibrary{arrows}
\ifdefined\isarxiv
\usepackage[margin=1in]{geometry}
\fi
\graphicspath{{./figs/}}

\usepackage{textcomp}
\usepackage{multirow}
\ifdefined\isarxiv
\newtheorem{theorem}{Theorem}[section]
\newtheorem{lemma}[theorem]{Lemma}
\newtheorem{definition}[theorem]{Definition}

\newtheorem{corollary}[theorem]{Corollary}
\newtheorem{conjecture}[theorem]{Conjecture}

\newtheorem{observation}[theorem]{Observation}
\newtheorem{remark}[theorem]{Remark}
\newtheorem{claim}[theorem]{Claim}

\fi
\newtheorem{fact}[theorem]{Fact}

\newcommand{\wt}{\widetilde}
\newcommand{\ov}{\overline}

\newcommand{\R}{\mathbb{R}}

\renewcommand{\d}{\mathrm{d}}

\newcommand{\diag}{\textrm{diag}}
\renewcommand{\d}{\mathrm{d}}

\DeclareMathOperator*{\E}{{\mathbb{E}}}

\DeclareMathOperator*{\C}{\mathbb{C}}
\DeclareMathOperator{\supp}{supp}

\DeclareMathOperator{\rank}{rank}

\DeclareMathOperator{\tr}{tr}

\SetAlgoProcName{Procedure}{Procedure}

\definecolor{mygreen}{RGB}{80,180,0}
\definecolor{b2}{RGB}{51,153,255}

\newcommand{\Ruizhe}[1]{{\color{red}[Ruizhe: #1]}}
\newcommand{\Zhao}[1]{{\color{b2}[Zhao: #1]}}

\newcommand{\nc}{\newcommand}

\nc{\nnl}{\nn \\ &}  
\nc{\fot}{\frac{1}{2}} 
\nc{\oo}[1]{\frac{1}{#1}} 
\newcommand{\ben}{\begin{enumerate}}
\newcommand{\een}{\end{enumerate}}
\nc{\mc}{\mathcal}

\nc{\onenorm}[1]{\L\| #1 \R\|_1} 

\title{Hyperbolic Concentration, Anti-concentration, and Discrepancy}
\ifdefined\isarxiv
\author{
Zhao Song\thanks{\texttt{zsong@adobe.com}. Adobe Research.} 
\and
Ruizhe Zhang\thanks{\texttt{ruizhe@utexas.edu}. Department of Computer Science, University of Texas at Austin.}
}
\date{}
\else
\author{Zhao Song}{Adobe Research, Seattle, WA, USA}{zsong@adobe.com}{}{}
\author{Ruizhe Zhang}{The University of Texas at Austin, Austin, TX, USA}{ruizhe@utexas.edu}{}{This work was supported by Dana Moshkovitz's NSF Grant CCF-1648712.}
\authorrunning{Z. Song and R. Zhang}
\Copyright{Zhao Song and Ruizhe Zhang}
\relatedversiondetails{Full Version}{https://arxiv.org/abs/2008.09593} 
\acknowledgements{
We thank the anonymous reviewers for helpful comments. The authors would like to thank Petter Br{\"a}nd{\'e}n and James Renegar  for many useful discussions about the literature of hyperbolic polynomials. The authors would like to thank Yin Tat Lee and James Renegar, Scott Aaronson  for encouraging us to work on this topic. The authors would like to thank Dana Moshkovitz for giving comments on the draft.}

\EventEditors{Amit Chakrabarti and Chaitanya Swamy}
\EventNoEds{2}
\EventLongTitle{Approximation, Randomization, and Combinatorial Optimization. Algorithms and Techniques (APPROX/RANDOM 2022)}
\EventShortTitle{\mbox{\scriptsize APPROX/RANDOM 2022}}
\EventAcronym{APPROX/RANDOM}
\EventYear{2022}
\EventDate{September 19--21, 2022}
\EventLocation{University of Illinois, Urbana-Champaign, USA (Virtual Conference)}
\EventLogo{}
\SeriesVolume{245}
\ArticleNo{10}
\fi
\begin{document}
\ifdefined\isarxiv
\begin{titlepage}
  \maketitle
  \begin{abstract}
  Chernoff bound is a fundamental tool in theoretical computer science. It has been extensively used in randomized algorithm design and stochastic type analysis.  Discrepancy theory, which deals with finding a bi-coloring of a set system such that the coloring of each set is balanced, has a huge number of applications in approximation algorithms design. Chernoff bound [Che52] implies that a random bi-coloring of any set system with $n$ sets and $n$ elements will have discrepancy $O(\sqrt{n \log n})$ with high probability, while the famous result by Spencer [Spe85] shows that there exists an $O(\sqrt{n})$ discrepancy solution. 

The study of hyperbolic polynomials dates back to the early 20th century when used to solve PDEs by G{\aa}rding [G{\aa}r59]. In recent years, more applications are found in control theory, optimization, real algebraic geometry, and so on. In particular, the breakthrough result by Marcus, Spielman, and Srivastava [MSS15] uses the theory of hyperbolic polynomials to prove the Kadison-Singer conjecture [KS59], which is closely related to discrepancy theory.  

In this paper, we present a list of new results for hyperbolic polynomials:

\begin{itemize}
    \item We show two nearly optimal hyperbolic Chernoff bounds: one for Rademacher sum of arbitrary vectors and another for random vectors in the hyperbolic cone.
    \item We show a hyperbolic anti-concentration bound.
    \item We generalize the hyperbolic Kadison-Singer theorem [Br{\"a}18] for vectors in sub-isotropic position, and prove a hyperbolic Spencer theorem for any constant hyperbolic rank vectors.
\end{itemize}

The classical matrix Chernoff and discrepancy results are based on determinant polynomial which is a special case of hyperbolic polynomials. To the best of our knowledge, this paper is the first work that shows either concentration or anti-concentration results for hyperbolic polynomials. We hope our findings provide more insights into hyperbolic and discrepancy theories.

  \end{abstract}
  \thispagestyle{empty}
\end{titlepage}
\else
 \maketitle
 \begin{abstract}
 
 \end{abstract}
\fi


\setcounter{page}{1} 
\pagenumbering{arabic} 
\section{Introduction}

The study of concentration of sums of independent random variables dates back to Central Limit Theorems, and hence to de Moivre and Laplace, while modern concentration bounds for sums of random variables were probably first established by Bernstein \cite{b24} in 1924. An extremely popular variant now known as Chernoff bounds was introduced by Rubin and published by Chernoff \cite{c52} in 1952.

Hyperbolic polynomials are real, multivariate homogeneous polynomials $p(x)\in \R[x_1,\dots,x_n]$, and we say that $p(x)$ is hyperbolic in direction $e\in \R^n$ if the univariate polynomial $p(te-x) = 0$ for any $x$ has only real roots as a function of $t$ (counting multiplicities). The study of hyperbolic polynomials was first proposed by G{\aa}rding in \cite{g51} and has been extensively studied in the mathematics community \cite{g59,g97,bgls01,r06}. 
Some examples of hyperbolic polynomials are as follows:
\begin{itemize}
    \item Let $h(x)=x_1x_2\cdots x_n$. It is easy to see that $h(x)$ is hyperbolic with respect to any vector $e\in \R_+^n$.
    \item Let $X = (x_{i,j})_{i,j=1}^n$ be a symmetric matrix where $x_{i,j} = x_{j,i}$ for all $1\leq i,j\leq n$. The determinant polynomial $h(x) = \det(X)$ is hyperbolic with respect to $\wt{I}$, the identity matrix $I$ packed into a vector. Indeed, $h(t \wt{I}-x)=\det(tI-X)$, the characteristic polynomial of the symmetric matrix $X$, has only real roots by the spectral theorem.
    \item Let $h(x) = x_1^2 - x_2^2-\cdots - x_n^2$. Then, $h(x)$ is hyperbolic with respect to $e=\begin{bmatrix}1&0&\cdots & 0\end{bmatrix}^\top$.
\end{itemize}
\begin{figure}[h]
    \centering
    \includegraphics[width=0.70 \textwidth]{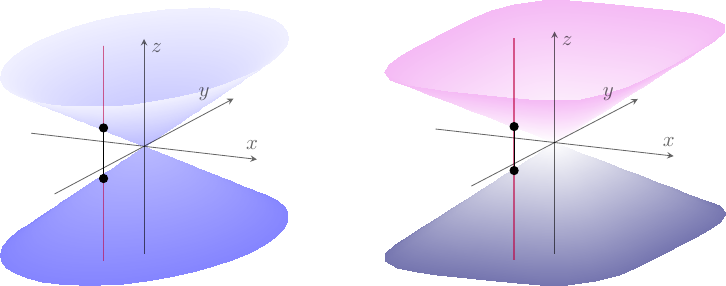}%
    \caption{\small The function on the left is $h(x,y,z)=z^2-x^2-y^2$, which is hyperbolic with respect to $e=\begin{bmatrix}0&0& 1\end{bmatrix}^\top$, since any line in this direction always has two intersections, corresponding to the two real roots of $h(-x,-y,t-z)=0$. The function on the right is $g(x,y,z) =z^4-x^4-y^4$, which is \emph{not} hyperbolic with respect to $e$, since it only has 2 intersections but the degree is 4.}
    \label{fig:hyper_poly}
\end{figure}
Inspired by the eigenvalues of matrix, we can define the hyperbolic eigenvalues of a vector $x$ as the real roots of $t\mapsto h(te-x)$, that is, $\lambda_{h,e}(x)=(\lambda_1(x),\dots,\lambda_d(x))$ such that $h(te-x)= h(e)\prod_{i=1}^d (t-\lambda_i(x))$ (see Fact~\ref{fac:hyper_factor}). In other words, the hyperbolic eigenvalues of $x$ are the zero points of the hyperbolic polynomial restricted to a real line through $x$. In this paper, we assume that $h$ and $e$ are fixed and we just write $\lambda(x)$ omitting the subscript. Furthermore, similar to the spectral norm of matrix, the \emph{hyperbolic spectral norm} of a vector $x$ can be defined as 
\begin{align}
    \|x\|_h = \max_{i\in [d]} |\lambda_i(x)|.
\end{align}

In this work, we study the concentration phenomenon of the roots of hyperbolic polynomials. More specifically, we consider the hyperbolic spectral norm of the sum of randomly signed vectors, i.e., $\|\sum_{i=1}^n r_i x_i\|_h$, where $r\in \{-1,1\}^n$ are uniformly random signs and $\{ x_1, x_2, \cdots, x_n \}$ are any fixed vectors in $\R^m$. This kind of summation has been studied in the following cases:
\begin{enumerate}
    \item \textbf{Scalar case:} $x_{i}\in \{-1, 1\}$ and the norm is just the absolute value, i.e., $|\sum_{i=1}^n r_i x_i|$,  the scalar version Chernoff bound \cite{c52} shows that
        $$\Pr_{r\sim \{-1, 1\}^n}\left[\left|\sum_{i=1}^n r_i x_i\right|>t\right]\leq 2\exp\left(-t^2/(2n)\right),$$
    corresponding to the case when $h(x)=x$ for $x\in \R$ and the hyperbolic direction $e=1$.
    \item \textbf{Matrix case:} $x_i$ are $d$-by-$d$ symmetric matrices and the norm is the spectral norm, i.e., $\|\sum_{i=1}^n r_i x_i\|$, the matrix Chernoff bound \cite{tro15} shows that 
        $$\Pr_{r\sim \{-1, 1\}^n}\left[\left\|\sum_{i=1}^n r_i x_i\right\|>t\right]\leq 2d\cdot \exp\left(-\frac{t^2}{2\left\|\sum_{i=1}^n x_i^2\right\|}\right),$$
    corresponding to $h(x) = \det(X)$ and $e= I$.
\end{enumerate}

We try to generalize these results to the hyperbolic spectral norm for any hyperbolic polynomial $h$, which is recognized as an interesting problem in this field by James Renegar \cite{r19}.  


\subsection{Our results}

In this paper, we can prove the following ``Chernoff-type'' concentration for hyperbolic spectral norm. We show that, when adding uniformly random signs to $n$ vectors, the hyperbolic spectral norm of their summation will concentrate with an exponential tail. 
\begin{theorem}[Nearly optimal hyperbolic Chernoff bound for Rademacher sum]\label{thm:intro_main}
Let $h$ be an $m$-variate, degree-$d$ hyperbolic polynomial with respect to a direction $e\in \R^m$.
Let $1\leq s \leq d$, $\sigma>0$. Given $x_1, x_2, \cdots, x_n \in \R^m$ such that $\rank(x_i)\leq s$ for all $i\in [n]$ and $\sum_{i=1}^n \|x_i\|_h^2 \leq \sigma^2$,
where $\rank(x)$ is the number of nonzero hyperbolic eigenvalues of $x$.
Then, we have
\begin{align*}
    \E_{ r \sim \{ \pm 1 \}^n }\left[\left\|\sum_{i=1}^n r_i x_i\right\|_h\right] \leq 2\sqrt{\log(s)}\cdot \sigma.
\end{align*}
Furthermore, for every $t>0$, and for some fixed constant $c>0$,
\begin{align*}
    \Pr_{ r \sim \{\pm 1\}^n } \left[ \left\| \sum_{i=1}^n r_i x_i \right\|_h > t \right]\leq 2\exp\left(-  \frac{c t^2}{\sigma^2 \log (s+1)}\right).
\end{align*}

\end{theorem}

We discuss the optimality of Theorem~\ref{thm:intro_main} in different cases:
\begin{itemize}
    \item \textbf{Degree-1 case: } When the hyperbolic polynomial's degree $d=s=1$, the hyperbolic polynomial is $h(z)=z$. Then, we have $\|x\|_h = |x|$ and we get the the Hoeffding's inequality \cite{hoe94}:
    $$\Pr_{r\sim \{\pm 1\}^n}\Big[ \Big|\sum_{i=1}^n r_i x_i\Big|>t \Big]\leq \exp\Big(- \Omega\Big(t^2/(\sum_{i=1}^n x_i^2)\Big)\Big).$$ It implies that our result is optimal in this case. 
    \item \textbf{A special degree-2 case:} $h(z)=z_1^2-z_2^2-\cdots -z_m^2$. Let $v_1,\dots,v_n$ be any $(d-1)$-dimensional vectors. Then, we define $x_i:= \begin{bmatrix}0 & v_i\end{bmatrix}\in \R^d$ for $i\in [n]$. We know that $\|x_i\|_h=\|v_i\|_2$, and Theorem~\ref{thm:intro_main} gives the following result:
    $$\Pr_{r\sim \{\pm 1\}^n}\Big[\Big\|\sum_{i=1}^n r_i v_i\Big\|_2>t\Big]\leq \exp(-\Omega(t^2/\sigma^2)),$$ where $\sigma^2:=\sum_{i=1}^n \|v_i\|^2$, which recovers the dimension-free vector-valued Bernstein inequality~\cite{m17}. 
    \item \textbf{Constant degree case: } When $d>1$ is a constant, consider $h$ being the determinant polynomial of $d$-by-$d$ matrix. Since $s\leq d=O(1)$, we can show that $\sigma=( \sum_{i=1}^n \|x_i\|^2 )^{1/2} = \Theta(\|\sum_{i=1}^n x_i^2\|^{1/2})$, and Theorem~\ref{thm:intro_main} exactly recovers the matrix Chernoff bound \cite{tro15}, which implies that our result is also optimal in this case. 
    \item \textbf{Constant rank case: }
    When all the vectors have constant hyperbolic rank, we still take $h=\det(X)$, but $X_1,\dots,X_n$ are constant rank matrices with arbitrary dimension. In this case, we can obtain a dimension-free matrix concentration inequality:
    $$\Pr_{r\sim \{\pm 1\}^n}\left[\left\|\sum_{i=1}^n r_i X_i\right\|>t\right]\leq 2\exp\left(-\Omega(t^2/\sigma^2)\right).$$
    It will beat the general matrix Chernoff bound \cite{tro15} when $\sigma$ is not essentially larger than $\|\sum_{i=1}^n X_i^2\|^{1/2}$. Thus, Theorem~\ref{thm:intro_main} is nearly optimal in this case. However, Theorem~\ref{thm:intro_main} is also sub-optimal in this case if we consider the high degree polynomial $h(z)=\prod_{i=1}^n z_i$, and $x_i=e_i\in \R^n$. Then, we have $\|x_i\|_h=1$, and $\|\sum_{i=1}^n r_i x_i\|_h=1$ for any $r\in \{\pm 1\}^n$. Therefore, the probability density function of the hyperbolic spectral norm of the Rademacher sum is a delta function\footnote{The delta function is defined as $\delta(x)=\begin{cases}1&\text{if}~x=1,\\0&\text{otherwise.}\end{cases}$} in this case. But our concentration result cannot characterize such a sharp transition.

\end{itemize}

Theorem~\ref{thm:intro_main} works for arbitrary vectors in $\R^m$. We also consider the maximum and minimum hyperbolic eigenvalues of the sum of random vectors in the hyperbolic cone, which is a generalization of the positive semi-definite (PSD) cone for matrices. Recall that for independent random PSD matrices ${\bf X}_1,\dots,{\bf X}_n$ with spectral norm at most $R$, let $\mu_{\max}:=\lambda_{\max}(\sum_i \E[{\bf X}_i])$. Then, matrix Chernoff bound for PSD matrices \cite{tro15} shows that 
$
    \Pr[\lambda_{\max} (\sum_{i} {\bf X}_i)\geq (1+\delta)\mu_{\max}] \leq~ de^ {-\Omega (\delta \mu_{\max})}
$
for any $\delta \geq 0$. The following theorem gives a hyperbolic version of this result:
\begin{theorem}[Hyperbolic Chernoff bound for random vectors in hyperbolic cone]\label{thm:hyper_chernoff_positive_intro}
Let $h$ be an $m$-variate, degree-$d$ hyperbolic polynomial with hyperbolic direction $e\in \R^m$. Let $\Lambda_+$ denote the hyperbolic cone\footnote{The hyperbolic cone is a set containing all vectors with non-negative hyperbolic eigenvalues. See Definition~\ref{def:hyper_cone} for the formal definition.} of $h$ with respect to $e$. Suppose $\mathsf{x}_1,\dots,\mathsf{x}_n$ are $n$ independent random vectors with supports in $\Lambda_+$ such that $\lambda_{\max}(\mathsf{x}_i)\leq R$ for all $i\in [n]$.
Define the mean of minimum and maximum eigenvalues as $\mu_{\min}:=\sum_{i=1}^n \E[\lambda_{\min}(\mathsf{x}_i)]$ and $\mu_{\max}:=\sum_{i=1}^n \E[\lambda_{\max}(\mathsf{x}_i)]$.

Then, we have
\begin{align*}
    &\Pr\left[\lambda_{\max}\left(\sum_{i=1}^n \mathsf{x}_i\right)\geq (1+\delta)\mu_{\max}\right] \leq~ d\cdot \left(\frac{(1+\delta)^{1+\delta}}{e^{\delta}}\right)^{-\mu_{\max}/R}~~\forall \delta\geq 0,\\
    &\Pr\left[\lambda_{\min}\left(\sum_{i=1}^n \mathsf{x}_i\right)\leq (1-\delta)\mu_{\min}\right] \leq~ d\cdot \left(\frac{(1-\delta)^{1-\delta}}{e^{-\delta}}\right)^{-\mu_{\min}/R}~~\forall \delta\in [0,1].
\end{align*}
\end{theorem}

\subsection{Hyperbolic anti-concentration}
Anti-concentration is an interesting phenomenon in probability theory, which studies the opposite perspective of concentration inequalities. A simple example is the standard Gaussian random variable, which has probability at most $O(\Delta)$ for being in any interval of length $\Delta$. For Rademacher random variables $x\sim \{\pm 1\}^d$, the celebrated Littlewood-Offord theorem \cite{lo43} states that for any degree-$1$ polynomial $p(x) = \sum_{i=1}^d a_i x_i$ with $|a_i|\geq 1$, the probability of $p(x)$ in any length-1 interval is at most $O(\frac{\log d}{\sqrt{d}})$. Later, the theorem was improved to $O(\frac{1}{\sqrt{d}})$ by Erd\"{o}s \cite{e45}, and generalized to higher degree polynomials by \cite{ctv06,rv13,mnv17}. From a geometric prospective, the Littlewood-Offord theorem says that the maximum fraction of hypercube points that lay in the boundary of a halfspace $\mathbf{1}_{\langle a, x\rangle\leq \theta}$ with $|a_i|\geq 1$ for $i\in [d]$ is at most $O(\frac{1}{\sqrt{d}})$. \cite{ost19} extended this result from half-space to polytope and \cite{ay21} further extended to positive spectrahedron. 

Following this line of research, we prove the following hyperbolic anti-concentration theorem, which shows that the hyperbolic spectral norm of Rademacher sum of vectors in the hyperbolic cone cannot concentrate within a small interval.

\begin{theorem}[Hyperbolic anti-concentration theorem, informal]\label{thm:anti-concen_intro}
Let $h$ be an $m$-variate degree-$d$ hyperbolic polynomial with hyperbolic direction $e\in \R^m$. Let $\{x_i\}_{i\in [n]}\subset \Lambda_+$ be a sequence of vectors in the hyperbolic cone such that $\lambda_{\max}(x_i)\leq \tau$ for all $i\in [n]$ and  $\sum_{i=1}^n \lambda_{\min}(x_i)^2\geq 1$. 

Then, for any $y\in \R^m$ and any $\Delta \geq 20\tau \log d$, we have
\begin{align*}
    \Pr_{\epsilon\sim \{-1,1\}^n}\left[\lambda_{\max}\left(\sum_{i=1}^n \epsilon_i x_i-y\right) \in [-\Delta, \Delta] \right]\leq O(\Delta).
\end{align*}
\end{theorem}

From the geometric viewpoint, we can define a ``positive hyperbolic-spectrahedron'' as the space $\{\alpha \in \R^n: \lambda_{\max}(\alpha_1x_1+\cdots +\alpha_nx_n-y)\leq 0\}$, where $x_1,\dots,x_n$ are in the hyperbolic cone. Then, Theorem~\ref{thm:anti-concen_intro} states that the hyperbolic spectral norm of a positive hyperbolic-spectrahedron cannot be concentrated in a small region.

\subsection{Hyperbolic discrepancy theory}\label{sec:app_discrepancy}
Hyperbolic polynomial is an important tool in the discrepancy theory, which is an important subfield of combinatorics, with many applications in theoretical computer science. Following Meka's blog post \cite{m14}, by combining scalar version Chernoff bound and union bound, we can easily prove that, for any $n$ vectors $x_1,\dots,x_n\in \{-1,1\}^n$, there exists $r\in \{-1,1\}^n$ such that $|\langle r, x_i\rangle|\leq O(\sqrt{n\log n})$ for every $i\in [n]$. In a celebrated result ``Six Standard Deviations Suffice'', Spencer showed that it can be improved to $|\langle r, x_i\rangle|\leq 6\sqrt{n}$ \cite{spe85}. 

For the matrix case, by the matrix Chernoff bound, it follows that for any symmetric matrix $X_1,\dots,X_n\in \R^{d\times d}$ with $\|X_i\|\leq 1$, for uniformly random signs $r\in \{-1,1\}^n$, with high probability, $\left\|\sum_{i=1}^n r_i X_i\right\|\leq O(\sqrt{\log (d)n})$.

An important open question is, can we shave the $\log (d)$ factor for \emph{some} choice of the signs?
\begin{conjecture}[Matrix Spencer Conjecture]\label{conj:matrix_spencer}
For any symmetric matrices $X_1,\dots,X_n\in \R^{d\times d}$ with $\|X_i\|\leq 1$, there exist signs $r\in \{-1, 1\}^n$ such that $\|\sum_{i=1}^n r_i X_i\|=O(\sqrt{n})$.
\end{conjecture}
The breakthrough paper by Marcus, Spielman and Srivastava \cite{mss15} proved the famous Kadison-Singer conjecture \cite{ks59}, which was open for more than half of a century.
\begin{theorem}[Kadison-Singer, \cite{ks59, mss15}]\label{thm:kadison_singer}
Let $k \geq 2$ be an integer and $\epsilon$ a positive real number. Let $x_1, \dots , x_n \in \C^m$ such that $\|x_ix_i^*\| \leq \epsilon ~~\forall i\in [n]$, and $\sum_{i=1}^n x_ix_i^* = I$. 
Then, there exists a partition $S_1\cup S_2\cup \cdots \cup S_k=[n]$ such that $\|\sum_{i\in S_j} x_ix_i^* \| \leq (\frac{1}{\sqrt{k}}+\sqrt{\epsilon} )^2\quad \forall j\in [k]$.
\end{theorem}

The Kadison-Singer theorem implies that for rank-1 matrices $X_1,\dots,X_n$ with $\|X_i\|\leq \epsilon$ in isotropic position\footnote{Isotropic means $X_1+\cdots + X_n = I$.}, there exists a choice of $r\in \{-1, 1\}^n$ such that $\|\sum_{i=1}^n r_i X_i\|\leq O(\sqrt{\epsilon})$.\footnote{For more details and consequences of the Kadison-Singer theorem, we refer the readers to \cite{ct16,ms16}.}

Theorem~\ref{thm:kadison_singer} can be generalized for higher rank matrices by Cohen \cite{c16} and Br{\"a}nd{\'e}n \cite{b18} independently. However, their results still need the isotropic condition. On the other hand, Kyng, Luh, and Song \cite{kls19} proved a stronger version of rank-1 matrix Spencer theorem  (Conjecture~\ref{conj:matrix_spencer}) by showing that when the spectral norm of the sum of the squared matrices (the variance of the random matrices) is bounded, the matrix discrepancy upper bound is at most four deviations. Formal theorem statements will be presented in Section~\ref{sec:hy_ks}.

Similar to the scalar and matrix cases, the discrepancy theory can be further generalized to the hyperbolic spectral norm. Br{\"a}nd{\'e}n \cite{b18} proved a hyperbolic Kadison-Singer theorem, which generalizes Theorem~\ref{thm:kadison_singer} to the hyperbolic spectral norm and vectors with arbitrary rank and in isotropic condition. Our first result relaxes the isotropic condition to sub-isotropic:

\begin{theorem}[Hyperbolic Kadison-Singer with sub-isotropic condition, informal]\label{thm:hyperbolic_ks_sub_intro}
Let $k \geq 2$ be an integer and $\epsilon,\sigma>0$. Suppose $h$ is hyperbolic with respect to $e \in \R^m$, and let $x_1, \dots , x_n$ be $n$ vectors in the hyperbolic cone such that
\begin{align}\label{eq:subisotropic_intro}
    \tr_h[x_i] \leq \epsilon~~\forall i\in [n], ~\text{and}\quad \Big\|\sum_{i=1}^n x_i\Big\|_h \leq \sigma.
\end{align}
where $\tr_h[x]:=\sum_{i=1}^d \lambda_i(x)$.
Then, there exists a partition $S_1 \cup S_2 \cup \cdots \cup S_k = [n]$ such that for all $j\in [k]$,
\begin{align*}
    \Bigg\| \sum_{i\in S_j} x_i\Bigg\|_h \leq  \left(\sqrt{\epsilon}+\sqrt{\sigma/k}\right)^2.
\end{align*}
\end{theorem}

Theorem~\ref{thm:hyperbolic_ks_sub_intro} implies the high rank case of \cite{mss15} result (Theorem~\ref{thm:kadison_singer}) without the isotropic condition. We note that there is a naive approach to relax the isotropic condition in \cite{mss15,b18}'s results by adding several small dummy vectors to make the whole set in isotropic position. (See \cite{ove15} for more details.) However, Theorem~\ref{thm:hyperbolic_ks_sub_intro} is slightly better than this approach, since the naive approach will increase the number of vectors which results in a worse bound. 
\ifdefined\isarxiv
(See Remark~\ref{rmk:compare_Branden} for more details.)
\fi

Theorem~\ref{thm:hyperbolic_ks_sub_intro} also implies the following hyperbolic discrepancy result:
\begin{corollary}[Hyperbolic discrepancy for sub-isotropic vectors]\label{cor:branden_discrepancy}
Let $0<\epsilon\leq \frac{1}{2}$. Suppose $h\in \R[z_1,\dots,z_m]$ is hyperbolic with respect to $e\in \R^m$, and let $x_1, \dots , x_n\in \Lambda_+(h,e)$ that satisfy Eq.~\eqref{eq:subisotropic_intro}.
Then, there exist signs $r\in \{-1,1\}^n$ such that
\begin{align*}
    \left\|\sum_{i=1}^n r_i x_i\right\|_h\leq 2\sqrt{\epsilon(2\sigma-\epsilon)}.
\end{align*}
\end{corollary}

We note that this result is incomparable with \cite{kls19} due to the following reasons: 1) \cite{kls19} only works for rank-1 matrices while our result holds for arbitrary rank vectors in the hyperbolic cone; 2) the upper bound of \cite{kls19} depends on $\|\sum_{i=1}^n X_i^2\|^{1/2}$ while our result depends on the hyperbolic trace and spectral norm of the sum of vectors.

To obtain a hyperbolic discrepancy upper bound for arbitrary vectors (as in the case of Conjecture~\ref{conj:matrix_spencer}), we can apply hyperbolic Chernoff bound (Theorem~\ref{thm:intro_main}) and get the following discrepancy result which holds with high probability:

\begin{corollary}\label{Cor:hyper_discrep_high_prob}
Let $h$ be a degree-$d$ hyperbolic polynomial with respect to $e\in \R^m$. We are given vectors $x_1,x_2,\cdots, x_n \in \R^m$ such that $\|x_i\|_h\leq 1$ and $\rank(x_i)\leq s$ for all $i \in [n]$ and some $s\in \mathbb{N}_+$. Then for uniformly random signs $r\sim \{-1, 1\}^n$, 
\begin{align*}
    \left\| \sum_{i=1}^n r_i x_i \right\|_h \leq O(\sqrt{n \log (s+1)})
\end{align*}
holds with probability at least $0.99$.
\end{corollary}

This result may not be tight when the ranks of the input vectors are large. It is thus interesting to study whether we can do better to improve the $\sqrt{\log d}$ factor in the non-constructive case. We thus conjecture the following hyperbolic discrepancy bound:

\begin{conjecture}[Hyperbolic Spencer Conjecture]\label{conj:hyperbolic_spencer}
We are given vectors $x_1,x_2,\cdots, x_n \in \R^m$ and a degree $d$ hyperbolic polynomial $h\in\R[z_1,\dots,z_m]$ with respect to $e\in \R^m$, where $\|x_i\|_h\leq 1$ for all $i \in [n]$. Then, there exist signs $r\in \{-1, 1\}^n$, such that
\begin{align*}
    \left\| \sum_{i=1}^n r_i x_i \right\|_h \leq O(\sqrt{n}).
\end{align*}
\end{conjecture}
Note that Conjecture~\ref{conj:hyperbolic_spencer} is more general than the matrix Spencer conjecture (Conjecture~\ref{conj:matrix_spencer}). And for constant degree $d$ or constant maximum rank $s$, this conjecture is true by Corollary~\ref{Cor:hyper_discrep_high_prob}.

\subsection{Related work}

\ifdefined\isarxiv
\paragraph{Chernoff-type bounds}
\else
\paragraph*{Chernoff-type bounds}
\fi
There is a long line of work generalizing the classical scalar Chernoff-type bounds to the matrix Chernoff-type bound \cite{r99,aw02,rv07, t12,mjc14,glss18,ks18,nrr20,aby20,jll20}.  
\cite{r99,rv07} showed a Chernoff-type concentration of spectral norm of matrices which are the outer product of two random vectors. \cite{aw02} first used Laplace transform and Golden-Thompson inequality \cite{g65,t65} to prove a Chernoff bound for general random matrices. It was improved by \cite{t12} and \cite{o09} independently. 
\cite{mjc14} proved a series of matrix concentration results via Stein's method of exchangeable pairs. Our work further extends this line of research from matrix to hyperbolic polynomials and can fully recover the result of \cite{aw02}. On the other hand, \cite{glss18} showed an expander matrix Chernoff bound.
\cite{ks18} prove a new matrix Chernoff bound for Strongly Rayleigh distributions.

\ifdefined\isarxiv
\paragraph{Hyperbolic polynomials}
\else
\paragraph*{Hyperbolic polynomials}
\fi
The concept of hyperbolic polynomials was originally studied in the field of partial differential equations \cite{g51, h83, k95}. G{\"u}ler \cite{g97} first studied the hyperbolic optimization (hyperbolic programming), which is a generalization of LP and SDP. Later, a few algorithms \cite{r06,mt14, rs14, ren16,np18, ren19} were designed for hyperbolic programming. On the other hand, a lot of recent research focused on the equivalence between hyperbolic programming and SDP, which is closely related to the ``Generalized Lax Conjecture'' and its variants \cite{hv07,lpr05, b14,kpv15,sau18,ami19,rrsw19}. In addition to the hyperbolic programming, hyperbolic polynomial is a key component in resolving Kadison-Singer problem \cite{mss15, b18} and constructing bipartite Ramanujan graphs \cite{mss18}. Gurvits \cite{g06,g07} proved some Van der Waerden/Schrijver-Valiant like conjectures for hyperbolic polynomials, giving sharp bounds for the capacity of polynomials. \cite{s19} gave an approach to certify the non-negativity of polynomials via hyperbolic programming, generalizing the Sum-of-Squares method.

\ifdefined\isarxiv
\paragraph{Discrepancy theory}
\else
\paragraph*{Discrepancy theory}
\fi
For discrepancy theory, we give a few literature in Section~\ref{sec:app_discrepancy} and we provide more related work here. For Kadison-Singer problem, after the breakthrough result \cite{mss15}, Anari and Oveis Gharan \cite{ag14} generalized it for Strongly Rayleigh distributions. Alishahi and Barzegar \cite{ab20} extended the ``paving conjecture''
to real stable polynomials\footnote{A polynomial is real stable if it is hyperbolic with respect to every $e\in \R^n_>0$.}. 
Zhang and Zhang \cite{zz21} further relaxed the determinant polynomial in \cite{ag14} and \cite{kls19} to homogeneous real-stable polynomials.  More recently, \cite{hrs21,djr21} proved some special cases of the matrix Spencer conjecture.
For algorithmic results, Bansal \cite{b10} proposed the first constructive version of partial coloring for discrepancy minimization. Based on this work, more approaches \cite{lm15, r17,lrr17,es18, bdgl18,dntt18} were discovered in recent years. 
For applications of the discrepancy theory, \cite{ag14,ag15} used the Strongly Rayleigh version of Kadison-Singer theorem to improve the integrality gap of the Asymmetric Traveling Salesman Problem. \cite{lz20} used the rank-1 matrix Spencer theorem in \cite{kls19} to obtain a two-sided spectral rounding result. For more applications, we refer to the excellent book by Matousek \cite{m09}.


\subsection{Technique overview}
In this section, we provide a proof overview of our results. We first show how prove hyperbolic Chernoff bounds by upper bounding each polynomial moment. After that, we show how to apply our new concentration inequality to prove hyperbolic anti-concentration. Finally, we show how to relax the isotropic condition in \cite{b18},  and also how to get a more general discrepancy result via hyperbolic concentration. 

\subsubsection{Our technique for hyperbolic Chernoff bound for Rademacher sum}
The main idea of our proof of hyperbolic Chernoff bound is to upper bound the polynomial moments. 

By definition, the hyperbolic spectral norm of $X$ is the $\ell_\infty$ norm of the eigenvalues $\lambda(X)$. Inspired by the proof of the matrix Chernoff bound by Tropp \cite{tro18}, we can consider the $\ell_{2q}$ norm of $\lambda(X)$, for $q\geq 1$. When the hyperbolic polynomial $h$ is the determinant polynomial, this norm is just the Schatten-$2q$ norm of matrices. For general hyperbolic polynomials, we define hyperbolic-$2q$ norm as $\|x\|_{h,2q}:=\|\lambda(x)\|_{2q}$. By the result of \cite{bgls01}, hyperbolic-$2q$ norm is actually a norm in $\R^m$. And the following inequality 
\ifdefined\isarxiv
(by Fact~\ref{fac:norm_h_bound_by_norm_h_q} and Lemma~\ref{lem:lyapunov}) 
\fi
shows the connection between a hyperbolic spectral norm and hyperbolic-$2q$ norm:
\begin{align*}
    \E_{ r \sim \{\pm 1\}^n }[\|X\|_h] \leq \Big(\E_{ r \sim \{\pm 1\}^n }\big[\|X\|_{h,2q}^{2q}\big]\Big)^{1/(2q)}.
\end{align*}

In order to compute $\|X\|_{h,2q}^{2q}=\sum_{i=1}^{\rank(X)} \lambda_i(X)^{2q}$, we use a deep result about hyperbolic polynomials: the Helton-Vinnikov Theorem \cite{hv07}, which proved a famous conjecture by Lax \cite{lax57}, to translate between hyperbolic polynomials and matrices. The theorem is stated as follows.

\begin{theorem}[\cite{hv07}]\label{thm:hv07}
Let $f \in \R[x, y, z]$ be hyperbolic with respect to $e = (e_1, e_2, e_3)\in \R^3$. Then there exist symmetric real matrices $A,B,C\in \R^{d\times d}$ such that $f = \det(xA + yB + zC)$ and $e_1A + e_2B + e_3C \succ 0$.
\end{theorem}
Gurvits \cite{gur04} proved a corollary
(Corollary~\ref{cor:hv_2_vars})
that for any $m$-variate hyperbolic polynomial $h$, and $x,y\in \R^m$, there exist two symmetric matrices $A,B\in \R^{d\times d}$ such that for any $a,b\in \R$, $\lambda(ax+by)=\lambda(aA+bB)$, where the left-hand side means the hyperbolic eigenvalues of the vector $ax+by$ and the right-hand side means the eigenvalues of the matrix $aA+bB$.

Therefore, we try to separate and consider one random variable $r_i$ at a time. We first consider the expectation over $r_1$. By conditional expectation, let $X_2:=\sum_{i=2}^n r_i x_i$ and we have
\begin{align*}
    \E_{ r \sim \{\pm 1\}^n }\big[\|X\|_{h,2q}^{2q}\big] = \E_{r_2,\dots,r_n\sim \{\pm 1\}}\left[\E_{r_1\sim \{\pm 1\}}\left[\|r_1x_1 + X_2\|_{h,2q}^{2q}\right]\right],
\end{align*}
By Corollary~\ref{cor:hv_2_vars}, 
there exist two matrices $A_1,B_1$ such that $\lambda(r_1x_1 + X_2)=\lambda(r_1 A_1 + B_1)$ holds for any $r_1$. 
And it follows that
\begin{align*}
    \E_{r_1\sim \{\pm 1\}}\left[\|r_1x_1 + X_2\|_{h,2q}^{2q}\right] = \E_{r_1\sim \{\pm 1\}}\left[\|r_1 A_1 + B_1\|_{2q}^{2q}\right].
\end{align*}
It becomes much easier to compute the expected Schatten-$2q$ norm of matrices. We can prove that
\begin{align*}
    \E_{ r \sim \{\pm 1\}^n }\big[\|X\|_{h,2q}^{2q}\big] \leq \sum_{k_1=0}^q {\binom{2q}{2k_1}\|x_1\|_h^{2k_1}\cdot  \E_{r_2,\dots,r_n}\left[\|X_2 \|_{h,2q-2k_1}^{2q-2k_1} \right]}.
\end{align*}
Now, we can iterate this process for the remaining expectation $\E_{r_2,\dots,r_n}\left[\|X_2 \|_{h,2q-2k_1}^{2q-2k_1} \right]$. After $n-1$ iterations, we get that
\begin{align}\label{eq:intro_final_exp}
    \left(\E_{r \sim \{\pm 1\}^n}\left[\left\|X\right\|_{h,2q}^{2q}\right]\right)^{1/(2q)} \leq \sqrt{2q-1}\cdot s^{1/(2q)}\cdot \sigma,
\end{align}
where $\sigma^2 =\sum_{i=1}^n \|x_i\|_h^2$ and $s$ is the maximum rank of $x_1,\dots,x_n$. Then, by taking $q:=\log(s)$ and $\|X\|_h \leq \|X\|_{h,2q}^{2q}$, we get the desired upper bound for the expectation $\E_{r\sim \{\pm 1\}^n}[\|\sum_{i=1}^n r_i x_i\|_h]$ in Theorem~\ref{thm:intro_main}.  

To obtain the concentration probability inequality, We can apply the result of Ledoux and Talagrand \cite{lt13} for the concentration of Rademacher sums in a normed linear space, which will imply:
\begin{align}\label{eq:intro_hyper_concentrate}
    \Pr_{ r \sim \{ \pm 1 \}^n }[ \| X \|_h > t ] \leq 2\exp \Big(-t^2 \Big/ \Big( 32\E_{ r \sim \{\pm 1\}^n }[\|X\|_h^2] \Big) \Big).
\end{align}
However, we need to verify that the hyperbolic spectral norm $\|\cdot\|_h$ is indeed a norm, which follows from the result of G{\aa}rding \cite{g59}.
Since by Khinchin-Kahane inequality
\ifdefined\isarxiv
(Theorem~\ref{thm:khinchin_kahane}),
\else
(\cite[Theorem A.16]{sz20})
\fi
the second moment of $\|X\|_h$ can be upper-bounded via the first moment. Hence, we can put our expectation upper bound into Eq.~\eqref{eq:intro_hyper_concentrate} and have
\begin{align*}
    \Pr_{ r \sim \{\pm 1\}^n } \left[ \left\| X \right\|_h > t \right]\leq C_1\exp\left(-\frac{C_2t^2}{\sigma^2 \log (s+1)}\right),
\end{align*}
for constants $C_1,C_2>0$, and hence Theorem~\ref{thm:intro_main} is proved. We defer the formal proof in 
\ifdefined\isarxiv
Section~\ref{sec:hyper_proof1}.
\else
the full version \cite[Section B]{sz20}.
\fi

\subsubsection{Our technique for hyperbolic Chernoff bound for positive vectors}
We can use similar techniques in the previous section to prove Theorem~\ref{thm:hyper_chernoff_positive_intro}.

For any random vectors $\mathsf{x}_1,\dots,\mathsf{x}_n\in \Lambda_+$, we may assume $\|\mathsf{x}_i\|_h\leq 1$. Using the Taylor expansion of the mgf, we can show that:
\begin{align}\label{eq:pos_mgf}
    \Pr\left[\lambda_{\max}\left(\sum_{i=1}^n \mathsf{x}_i\right)\geq t\right] \leq \inf_{\theta >0}~e^{-\theta t}\cdot\sum_{q\geq 0}\frac{\theta^q}{q!}\E\left[\left\|\sum_{i=1}^n \mathsf{x}_i\right\|_{h,q}^q\right].
\end{align}
Then, for the $q$-th moment, we separate $\mathsf{x}_1$ and $\sum_{i=2}^n \mathsf{x}_i$ and have
\begin{align*}
    \mathbb{E}_{\geq 1}\left[\left\|\sum_{i=1}^n x_i\right\|_{h,q}^q\right] = \mathbb{E}_{\geq 2} \mathbb{E}_1\left[\tr\left[(A_1+B_1)^q\right]\right],
\end{align*}
where $A_1$ and $B_1$ are two PSD matrices obtained via Gurvits's result
(Corollary~\ref{cor:hv_2_vars})
such that $A_1$ depends on $\mathsf{x}_1$ and $B_1$ depends on $\mathsf{x}_2,\dots,\mathsf{x}_n$. The next step is different from the case of Rademacher sum, since we cannot drop half of the terms by the distribution of $\mathsf{x_1}$. Instead, we can fully expand the matrix products in the trace and use Horn's inequality to upper bound the eigenvalue products. We have
\begin{align*}
    \mathbb{E}_{\geq 2} \mathbb{E}_1\left[\tr\left[(A(x_1)+B)^q\right]\right]\leq \mathbb{E}_{1} \left[\sum_{k_1=0}^q\binom{q}{k_1}\lambda_{\max}(\mathsf{x}_1)^{k_1}\cdot \mathbb{E}_{\geq 2} \left[\left\|\sum_{i=2}^n \mathsf{x}_i\right\|_{h,q-k_1}^{q - k_1}\right]\right].
\end{align*}
By repeating this process, we finally have
\begin{align*}
    \mathbb{E}\left[\left\|\sum_{i=1}^n \mathsf{x}_i\right\|_{h,q}^q\right]\leq  \E\left[ \sum_{\substack{k_1,\dots,k_n\geq 0\\k_1+\cdots+k_n=q}} \binom{q}{k_1,\dots,k_n} \prod_{i=1}^n \lambda_{\max}(\mathsf{x}_i)^{k_i}\cdot d \right]
    \leq d\cdot \E\left[\left(\sum_{i=1}^n \|\mathsf{x}_i\|_h\right)^q\right],
\end{align*}
where the first step follows from the $\E[\|{\sf x}_n\|_{h,k_n}^{k_n}]\leq d\cdot \lambda_{\max}({\sf x}_n)^{k_n}$.
Then, we put the above upper bound into Eq.~\eqref{eq:pos_mgf}, which gives:
\begin{align*}
    \Pr\left[\lambda_{\max}\left(\sum_{i=1}^n \mathsf{x}_i\right)\geq t\right] \leq \inf_{\theta >0}~e^{-\theta t}\cdot d\cdot \prod_{i=1}^n \E\left[e^{\theta\|\mathsf{x}_i\|_h}\right].
\end{align*}
Now, we use some similar calculations in the matrix case \cite{t12} to prove that
\begin{align*}
    \Pr\left[\lambda_{\max}\left(\sum_{i=1}^n \mathsf{x}_i\right)\geq t\right] \leq \inf_{\theta>0}~d\cdot \exp\left(-\theta t + (e^\theta - 1)\mu_{\max}\right).
\end{align*}
By taking $\theta := \log(t/\mu_{\max})$ and $t:=(1+\delta)\mu_{\max}$, we get that
\begin{align}\label{eq:tech_max_eig}
    \Pr\left[\lambda_{\max}\left(\sum_{i=1}^n \mathsf{x}_i\right)\geq (1+\delta)\mu_{\max}\right] \leq&~ d\cdot \left(\frac{(1+\delta)^{1+\delta}}{e^{\delta}}\right)^{-\mu_{\max}}
\end{align}

For the minimum eigenvalue case, we can define $\mathsf{x}'_i:=e-\mathsf{x}_i$ for $i\in [n]$. Then, by the property of hyperbolic eigenvalues
(Fact~\ref{fac:eigen_linear_trans})
and the assumption that $\|\mathsf{x}_i\|_h\leq 1$, we know that $\mathsf{x}'_i$ are also in the hyperbolic cone and $\lambda_{\max}(\mathsf{x}'_i)=1-\lambda_{\min}(\mathsf{x}'_i)$. Therefore, we can obtain the Chernoff bound for the minimum eigenvalue of $\mathsf{x}$ by applying Eq.~\eqref{eq:tech_max_eig} with $\mathsf{x}_i'$. We defer the formal proof in
\ifdefined\isarxiv
Section~\ref{sec:hyper_proof2}.
\else
the full version \cite[Section C]{sz20}.
\fi
\subsubsection{Our technique for hyperbolic anti-concentration}
In this part, we will show how to prove the hyperbolic anti-concentration result (Theorem~\ref{thm:anti-concen_intro}) via the hyperbolic Chernoff bound for vectors in the hyperbolic cone (Theorem~\ref{thm:hyper_chernoff_positive_intro}).

In \cite{ost19}, they studied the unate functions on hypercube $\{-1, 1\}^n$, which is defined as the function being increasing or decreasing with respect to any one of the coordinates. Then, they showed that the Rademacher measure of a unate function is determined by the expansion of its indicator set in hypercube. In particular, for the maximum hyperbolic eigenvalue, it is easy to see that the indicator function $\left[\lambda_{\max}\left(\sum_{i=1}^n \epsilon_i x_i^j-y_j\right)\in [-\Delta, \Delta]\right]$ is unate when $x_i\in \Lambda_+$. Hence, we can show the anti-concentration inequality by studying the expansion in the hypercube, which by \cite{ay21}, is equivalent to lower-bound the minimum eigenvalue of each vector. However, for the initial input $x_i$, we only assume that $\sum_{i=1}^n\lambda_{\min}(x_i)^2\geq 1$, but we need a $\Omega(\frac{1}{\sqrt{\log d}})$ lower bound for each $x_i$ to prove the theorem. To amplify the minimum eigenvalue, we follow the proof in \cite{ay21} that uses a random hash function to randomly assign the input vectors into some buckets and considers the sum of the vectors in each bucket as the new input. They proved that the ``bucketing'' will not change the distribution. Then, we can use Theorem~\ref{thm:hyper_chernoff_positive_intro} to lower bound the minimum hyperbolic eigenvalue of each bucket, which is a sum of independent random vectors in the hyperbolic cone. Hence, we get that
\begin{align*}
    \Pr\left[\lambda_{\min}\left(\sum_{i=1}^n z_{i,j}x_i\right) \leq \Omega(\frac{1}{\sqrt{\log d}}) \right]\leq \frac{1}{10},
\end{align*}
which $z_{i,j}\in \{0, 1\}$ is a random variable indicating that $x_i$ is hashed to the $j$-th bucket. Then, by the standard Chernoff bound for negatively associated random variables, we can prove that most of the buckets have large minimum eigenvalues, which concludes the proof of the hyperbolic anti-concentration theorem. We defer the formal proof in Section~\ref{sec:anticon}.

\subsubsection{Our technique for hyperbolic discrepancy}
To relax the isotropic condition in \cite{b18}, we basically follow their proof. The high-level idea is to construct a compatible family of polynomials\footnote{The compatible family of polynomials is closely related to the interlacing family in \cite{mss15,mss18}. 
\ifdefined\isarxiv
See Definition~\ref{def:compatible}.
\else
See \cite[Definition E.16]{sz20}.
\fi
} such that the probability in the hyperbolic Kadison-Singer problem (Theorem~\ref{thm:hyperbolic_ks_sub_intro}) can be upper-bounded by the largest root of the expected polynomial of the family, which can be further upper-bounded by the largest root of the mixed hyperbolic polynomial $h[v_1,\dots,v_n]\in \R[x_1,\dots,x_m,y_1,\dots,y_n]$, defined as $h[v_1,\dots,v_n]:=\prod_{i=1}^m (1-y_i D_{v_i}) h(x)$,
where $D_{v_i}$ is the directional derivative with respect to $v_i$. In particular, we can consider the roots of the linear restriction $h[v_1,\dots,v_n](te+\mathbf{1})\in \R[t]$. Then, using G{\aa}rding's result \cite{g59} on hyperbolic cone, we know that the largest root equals the minimum $\rho>0$ such that the vector $\rho e + \mathbf{1}$ is in the hyperbolic cone $\Gamma_+$ of $h[v_1,\dots,v_n]$, which can be upper-bounded via similar techniques in \cite{mss15,kls19} to iteratively add each vector $v_i$ while keeping the sum in the hyperbolic cone. Our key observation is that the proof in \cite{b18} essentially proved that
\begin{align*}
    \frac{\epsilon \mu e + \left(1-\frac{1}{n}\right)\delta \sum_{i=1}^n v_i }{1+\frac{\mu-1}{n}}+\mathbf{1}\in \Gamma_+
\end{align*}
holds for any vectors $v_i\in \Lambda_+$. Hence, once we assume that $\|\sum_{i=1}^n v_i\|_h\leq \sigma$, then by the convexity of the hyperbolic cone, we get that $\rho\leq \frac{\left(\epsilon \mu + \left(1-\frac{1}{n}\right)\delta\sigma\right) }{1+\frac{\mu-1}{n}}$, which will imply the upper bound in Theorem~\ref{thm:hyperbolic_ks_sub_intro}.
We defer the formal proof in
\ifdefined\isarxiv
Section~\ref{sec:hy_ks}.
\else
the full version \cite[Section E]{sz20}.
\fi

To obtain the discrepancy result for arbitrary vectors (Corollary~\ref{Cor:hyper_discrep_high_prob}), we can use the hyperbolic Chernoff bound for Rademacher sum (Theorem~\ref{thm:intro_main}) to derive the discrepancy upper bound. For any vectors $x_1,\dots,x_n$ with maximum rank $s$, by setting $t=O(\sigma\sqrt{\log s})$ in Theorem~\ref{thm:intro_main}, we get that $\left\| \sum_{i=1}^n r_i x_i \right\|_h \leq O(\sigma \sqrt{\log s})$
holds with high probability for uniformly random signs $r\sim \{\pm 1\}^n$.

\subsection{Discussion and Open problems}

In this paper, we initiate the study of concentration with respect to the hyperbolic spectral norm, and we generalize several classical concentration and anti-concentration results to the hyperbolic polynomial setting. Our results are closely related to the discrepancy theory and pseudorandomness. We provide some open problems in below.
\ifdefined\isarxiv
\paragraph{Tighter hyperbolic Chernoff bound?}
\else
\paragraph*{Tighter hyperbolic Chernoff bound?}
\fi
Our current result has a worse dependence on the variance $\sigma^2$ than the matrix Chernoff bound \cite{tro15}. Can we match the results when $h=\det(X)$? We note that there is a limitation for using the techniques like Golden-Thompson inequality and Lieb's theorem, which were used in \cite{o09, t12} to improve the original matrix Chernoff bound \cite{aw02}, to tighten our result. Because for any symmetric matrix $X$, we can define a mapping such that $\phi(X)$'s eigenvalues are the $p$-th power of $X$'s eigenvalues for any $p>0$, where the mapping is just $X^p$. However, we cannot find such a mapping for vectors with respect to the hyperbolic eigenvalues. Some new techniques may be required to get a hyperbolic Chernoff bound matching the matrix results.

\ifdefined\isarxiv
\paragraph{Resolving the hyperbolic Spencer conjecture?}
\else
\paragraph*{Resolving the hyperbolic Spencer conjecture?}
\fi

Inspired by the matrix Spencer conjecture (due to Meka \cite{m14}), we came up with a more general conjecture for hyperbolic discrepancy. Can we prove or disprove this conjecture? It is also interesting to study the connection between hyperbolic Spencer conjecture and the generalized Lax conjecture \cite{hv07,lpr05, b14,kpv15,sau18,ami19,rrsw19}. If we assume the matrix Spencer conjecture and the generalized Lax conjecture, can we prove the hyperbolic Spencer conjecture? On the other hand, in a very recent work by Reis and Rothvoss \cite{rr20}, they conjectured a weaker matrix Spencer by considering the Schatten-$p$ norm of matrices. We can also define such an $\ell_p$ version of the hyperbolic Spencer conjecture by looking at the $\ell_p$-norm of hyperbolic eigenvalues (the hyperbolic-$p$ norm). 
Any progress towards the $\ell_p$-hyperbolic Spencer conjecture will provide more insights in matrix and hyperbolic discrepancy theory. 

\ifdefined\isarxiv
\paragraph{Fooling hyperbolic cone?}
\else
\paragraph*{Fooling hyperbolic cone?}
\fi
One of the results in this paper is showing an anti-concentration inequality with respect to the hyperbolic spectral norm, which generalizes the results in \cite{ost19,ay21}. They actually combined the anti-concentration results with the Meka-Zuckerman \cite{mz13} framework to construct PRGs fooling polytopes/positive spectrahedrons. Hence, an open question in complexity theory and pseudorandomness is: can we apply the hyperbolic anti-concentration inequality to construct a PRG fooling positive hyperbolic-spectrahedrons, or even hyperbolic cones?

\paragraph*{Concentration of random tensors?}
Tensor concentration is another natural generalization of matrix concentration. Although there have been a large number of works on this problem \cite{l06,l11,al12,aw15,v20,alm21}, it is still unclear what is the optimal concentration bound for the Euclidean norm of random tensor $X\in \R^{n^d}$, even in the simple case when $X=x_1\otimes \cdots \otimes x_d$ for random vectors $x_1,\dots,x_d\in \R^n$. On the other hand, people also care about whether random tensors are well-conditioned, which is more related to TCS problems including tensor decompositions and learning Gaussian mixtures. The current results \cite{v20,a15,bcmv14} have a large gap between the matrix case. For these tensor concentration problems, is it possible to study them via hyperbolic polynomials and obtain tighter bounds?

\ifdefined\isarxiv
\newpage
{\hypersetup{linkcolor=black}
\tableofcontents
}
\fi
\ifdefined\isarxiv
\else
\bibliographystyle{plainurl}
\bibliography{ref}
\fi

\ifdefined\isarxiv
\newpage
\appendix
\section{Preliminaries}

\subsection{Notations}

For a vector $x$, we use $\| x \|_0$ to denote the number of non-zeros, use $\| x \|_1$ to denotes its $\ell_1$ norm, and use $\| x \|_p$ to denote its $\ell_p$ norm for $0<p\leq \infty$.

We use $r \in \{\pm 1 \}^n$ to denote $n$ i.i.d. random variables where each $r_i$ is $1$ with probability $1/2$ and $-1$ otherwise.

The general definition of semi-norm and norm is as follows:
\begin{definition}[Semi-norm and norm]
Let $\|\cdot \|:V\rightarrow \R$ be a nonnegative function on vector space $V$. We say $\|\cdot \|$ is a semi-norm if it satisfies the following properties:
For all $a\in \R$ and $x,y\in V$,
\begin{itemize}
    \item $\|x+y\|\leq \|x\|+\|y\|$;
    \item $\|ax\|=|a|\cdot \|x\|$.
\end{itemize}
If furthermore, $\|x\|=0$ implies $x=0$ the zero vector of $V$, then we say $\|\cdot \|$ is a norm.
\end{definition}

\begin{definition}[Normed linear space]
A normed linear space is a vector space over $\R$ or $\C$, on which a normed is defined.
\end{definition}


\subsection{Basic definitions of hyperbolic polynomials}

We provide the definition of hyperbolic polynomial.
\begin{definition}[Hyperbolic polynomial]
A homogeneous polynomial $h:\R\rightarrow \R$ is hyperbolic with respect to a vector $e\in \R^m$ if $h(e) \ne 0$, and for all $x \in \R^m$, the univariate polynomial $t \mapsto h(te - x)$ has only real zeros.
\end{definition}

The following fact shows how to factorize a hyperbolic polynomial, which easily follows from the homogeneity of the polynomial:
\begin{fact}[Hyperbolic polynomial factorization]\label{fac:hyper_factor}
For a degree-$d$ polynomial $h\in \R[z_1,\dots,z_m]$ hyperbolic with respect to $e\in \R^m$, we have 
\begin{align*}
    h(te-x)=h(e)\prod_{i=1}^d (t-\lambda_i(x))
\end{align*}
where $\lambda_1(x) \geq \lambda_2(x) \geq \cdots \geq \lambda_d(x)$ are real roots of $h(te-x)$.
\end{fact}

All the vectors with nonnegative hyperbolic eigenvalues form a cone, which is proved by G{\aa}rding \cite{g59}. It is a very important object related to the geometry of hyperbolic polynomials. The formal definition is as follows:
\begin{definition}[Hyperbolic cone]\label{def:hyper_cone}
For a degree $d$ hyperbolic polynomial $h$ with respect to $e\in \R^m$, its hyperbolic cone is
\begin{align*}
    \Lambda_+(e) := \{x: \lambda_d(x) \geq 0\}.
\end{align*}
The interior of $\Lambda_+^m$ is 
\begin{align*}
    \Lambda_{++}(e) := \{x: \lambda_d(x) > 0\}.
\end{align*}
\end{definition}
G{\aa}rding \cite{g59} showed the following fundamental properties of the hyperbolic cone:
\begin{theorem}[\cite{g59}]\label{thm:hyperbolic_cone_gar}
Suppose $h\in \R[z_1,\dots,z_m]$ is hyperbolic with respect to $e\in \R^n$. Then,
\begin{enumerate}
    \item $\Lambda_+(e),\Lambda_{++}(e)$ are convex cones.
    \item $\Lambda_++(e)$ is the connected component of $\{x\in \R^m: h(x)\ne 0\}$ which contains $e$.
    \item $\lambda_{\min}:\R^m\rightarrow \R$ is a concave function, and $\lambda_{\max}:\R^m\rightarrow \R$ is convex.
    \item If $e'\in \Lambda_{++}(e)$, then $h$ is also hyperbolic with respect to $e'$ and $\Lambda_{++}(e')=\Lambda_{++}(e)$.
\end{enumerate}
\end{theorem}

For simplicity, we may use $\Lambda_+$ and $\Lambda_{++}$ to denote $\Lambda_+(e),\Lambda_{++}(e)$, when $e$ is clear from context.
In this paper, we always assume that $e$ is any fixed vector in the hyperbolic cone of $h$. 

We define the trace, rank and spectral norm respect to hyperbolic polynomial $h$.
\begin{definition}[Hyperbolic trace, rank, spectral norm]
Let $h$ be a degree $d$ hyperbolic polynomial with respect to $e\in \R^m$. For any $x\in \R^m$,
\begin{align*}
    \tr_h [ x ] := \sum_{i=1}^d \lambda_i(x), \quad \rank(x) := \#\{i: \lambda_i(x)\ne 0\}, \quad \|x\|_h := \max_{i\in [d]}|\lambda_i(x)|=\max \{\lambda_1(x), -\lambda_d(x)\}.
\end{align*}
\end{definition}

We define the $p$ norm with respect to hyperbolic polynomial $h$.
\begin{definition}[$\| \cdot\|_{h,p}$ norm]
For any $p \geq 1$, we define the hyperbolic $p$-norm $\|\cdot \|_{h,p}$ defined as:
\begin{align*}
    \|x\|_{h,p} := \|\lambda(x)\|_{p} = \Big( \sum_{i=1}^d |\lambda_i(x)|^p \Big)^{1/p} \quad \forall x\in \R^m.
\end{align*}
\end{definition}

It has been shown that $\|\cdot\|_h$ and $\|\cdot \|_{h,p}$ are indeed norms:
\begin{theorem}[\cite{g59, b18, ren19}]\label{thm:hyperbolic_norm}
$\|\cdot\|_h$ is a semi-norm.

Furthermore, if $\Lambda_+$ is regular, i.e., $(\Lambda_+ \cap -\Lambda_+)=\{0\}$, then $\|\cdot\|_h$ is a norm on $\R^m$.
\end{theorem}

\begin{theorem}[\cite{bgls01}]
For any $p\geq 1$, $\|\cdot\|_{h,p}$ is a semi-norm. Moreover, if the hyperbolic cone $\Lambda_+$ is regular, then $\|\cdot\|_{h,p}$ is a norm.
\end{theorem}
\subsection{Basic properties of hyperbolic polynomials}

We state a fact for the eigenvalues $\lambda(\cdot)$ of degree-$d$ hyperbolic polynomial $h$.
\begin{fact}[\cite{bgls01}]\label{fac:eigen_linear_trans}
For all $i\in [d]$, 
\begin{align*}
\lambda_i( s \cdot x + t \cdot e )=\begin{cases}
s \cdot \lambda_i(x)+t, & \mathrm{~if~}s\ge 0;\\
s \cdot \lambda_{d-i}(x)+t, & \mathrm{~if~}s<0.
\end{cases}
\end{align*}
\end{fact}
Then, we show that the elementary symmetric sum-products of eigenvalues can be computed from the directional derivatives of the polynomial.
\begin{observation}[\cite{bgls01}]\label{obs:symmetric}
For a degree-$d$ hyperbolic polynomial $h$ with respect to $e$, we have
\begin{align*}
    h(te+x) = p(e) \cdot \prod_{i=1}^d (t+\lambda_i(x)) = \sum_{i=0}^d s_i(\lambda (x)) \cdot t^{d-i},
\end{align*}
where $\lambda(x)=(\lambda_1(x), \cdots, \lambda_d(x))$ are the hyperbolic eigenvalues of $x$ and $s_i : \R^d \rightarrow \R$ is the $i$-th elementary symmetric polynomial:
\begin{align*}
    s_i(y) := \begin{cases} 
        \sum_{S\in \binom{[d]}{i}} \prod_{j\in S} y_{j}, & \forall i \in [d];\\
        1 & \text{if}~ i = 0.
    \end{cases}
\end{align*}

Furthermore, for each $i \in \{0,1,\cdots, d\}$, 
\begin{align*}
    h(e) \cdot s_i(\lambda(x))=\frac{1}{(d-i)!} \cdot \nabla^{d-i} h(x) \underbrace{ [e,e,\dots, e] }_{(d-i) \mathrm{~terms}}.
\end{align*}
If $i \in [ d ]$, then $s_i \circ \lambda$ is hyperbolic with respect to $e$ of degree $i$.
\end{observation}

\begin{corollary}
$\tr[x]$ is a linear function.
\end{corollary}
\begin{proof}
By Observation~\ref{obs:symmetric}, we have
\begin{align*}
    \tr_h[x]=s_1(\lambda(x)) = \frac{1}{h(e)\cdot (d-1)!}\cdot \nabla^{d-1} h(x) [e,e,\dots, e].
\end{align*}
Since $h$ is of degree $d$, $\nabla^{d-1}h$ is a degree-1 polynomial. Hence, $\tr_h[x]$ is a linear function.
\end{proof}

\subsection{Concentration inequalities}
In general, for any normed linear space,
as mentioned in \cite{lt13}, we have the following concentration result:
\begin{theorem}[Theorem 4.7 in \cite{lt13}]\label{thm:banach_concentration}
Let $x_1,\dots,x_n\in \mathcal{B}$ be a fixed finite sequence in normed linear space $\mathcal{B}$. Let $X=\sum_{i=1}^n r_i x_i$, where $r_1,\dots,r_n$ are independent Rademacher random variables. Then, for every $t>0$,
\begin{align*}
    \Pr_{ r \sim \{ \pm 1 \}^n }[\|X\|_B>t]\leq 2\exp(-t^2/(32\E[\|X\|_B^2])).
\end{align*}
\end{theorem}




For matrices with Schatten-$p$ norm, the expectation of Schatten-$2p$ norm of Rademacher sum can be upper-bounded as follows.
\begin{theorem}[Theorem 3.1 in \cite{tj74}]\label{thm:matrix_p_norm}
Let $p\geq 1$. For a matrix $A$, we use $\| A \|_p$ to denote the Shatten-$p$ norm. 
For any fixed $X_1, X_2, \dots, X_n \in \R^{d \times d}$, and for independent Rademacher random variables $r_1, r_2, \dots, r_n$, we have
\begin{align*}
    \left( \E_{r \sim \{ \pm 1 \}^n } \left[ \left\|\sum_{i=1}^n r_i X_i \right\|_{2p}^{2p} \right] \right)^{1/(2p)}\leq \sqrt{2p-1} \cdot \left( \sum_{i=1}^n \| X_i \|_{2p}^2 \right)^{1/2}
\end{align*}
\end{theorem}

\subsection{Khinchin-Kahane inequality}

In any normed linear space, for any $p,q\geq 1$, the $p$-th moment and $q$-th moment of the norm of Rademacher sum are equivalent up to a constant, as shown in \cite{kah64}, which generalized the Khinchin inequality \cite{khi23}.

\begin{theorem}[\cite{kah64}; also in \cite{lo94,lt13, kr16}]\label{thm:khinchin_kahane}
For all $p, q \in [1, \infty)$, there exists a universal constant $C_{p,q} > 0$ depending only on $p, q$, such that for all choices of normed linear space $\mathcal{B}$, finite sets of vectors $x_1, x_2, \cdots , x_n \in \mathcal{B}$, and independent Rademacher variables $r_1, r_2, \cdots , r_n$,
\begin{align*}
    \left( \E_{ r \sim \{\pm 1\}^n }\left[\left\|\sum_{i=1}^n r_i x_i\right\|^q\right] \right)^{1/q}\leq C_{p,q}\cdot \left( \E_{ r \sim \{ \pm 1\}^n }\left[\left\|\sum_{i=1}^n r_i x_i\right\|^p\right] \right)^{1/p}.
\end{align*}
If moreover $1=p\leq q\leq 2$, then $C_{1,q}=2^{1-1/q}$ is optimal. If $q\in [1,\infty]$, then $C_{1,q}\leq \sqrt{q}$.

\end{theorem}

\subsection{Matrix analysis tools}
We state a Lemma for singular values of the product of matrices.
\begin{lemma}[General Horn inequality, Lemma 1.2 in \cite{tj74}]\label{lem:horn}
Let $A_1, \cdots, A_n \in \R^{d \times d}$ be symmetric matrices. Let $\sigma_1(A),\dots, \sigma_d(A)$ denote the singular values of $A$. Then, for each $k \in [d]$,
\begin{align*}
    \sum_{j=1}^k \sigma_j\left(\prod_{i=1}^n A_i\right) \leq \sum_{j=1}^k \prod_{i=1}^n \sigma_j(A_i).
\end{align*}
\end{lemma}

We state a Lemma which is implied by H\"{o}lder inequality.
\begin{lemma}[Lyapunov's inequality]\label{lem:lyapunov}
Let $0<r<s<\infty$ and $X$ be a random variable. Then,
\begin{align*}
    \E\left[|X|^r\right]\leq \left(\E\left[|X|^s\right]\right)^{r/s}.
\end{align*}
\end{lemma}

\subsection{Helton-Vinnikov Theorem}

We state a corollary of Helton-Vinnikov Theorem (Theorem~\ref{thm:hv07}), proved by Gurvits \cite{gur04}:
\begin{corollary}[Proposition 1.2 in \cite{gur04}]\label{cor:hv_2_vars}
Let $h(x)$ be a $m$-variable degree-$d$ hyperbolic polynomial. Then, for $x,y\in \R^m$, there exists two symmetric real matrices $A,B\in \R^{d\times d}$ such that for any $a,b\in \R$, the ordered eigenvalues $\lambda(ax+by)=\lambda(aA+bB)$.
\end{corollary}

This Corollary relates the hyperbolic eigenvalues of a vector $ax+by$ to the eigenvalues of matrix $aA+bB$, which allows us to study some properties of hyperbolic eigenvalues using results in matrix theory.
\section{Hyperbolic Chernoff bound for Rademacher sums}

In this section, we will prove the Chernoff bound for hyperbolic polynomials (Theorem~\ref{thm:main}). In Section~\ref{sec:sec3_prelim}, we provide some basic facts on the concentration of hyperbolic norm. Then, we prove the main result in Section~\ref{sec:hyper_proof1} and Section~\ref{sec:exp_hyper_2p_norm}.

\subsection{Preliminaries}\label{sec:sec3_prelim}
Recall that the hyperbolic spectral norm $\|\cdot\|_h$ is defined as:
\begin{align*}
    \|x\|_h := \|\lambda(x)\|_\infty.
\end{align*}

We should assume that the hyperbolic cone $\Lambda_{h,+}$ is regular. By Theorem~\ref{thm:hyperbolic_norm}, we know that $\|\cdot\|_h$ is a norm and $(\R^m, \|\cdot\|_h)$ is a normed linear space.
Applying the general concentration on normed linear space (Theorem~\ref{thm:banach_concentration}) to the $\|\cdot\|_h$ norm, and get the following result:
\begin{corollary}[Concentration of hyperbolic norm]\label{cor:hyperbolic_prob}
Let $X=\sum_{i=1}^n r_i x_i$, where $r_1, r_2, \cdots, r_n$ are independent Rademacher variables and $x_1, x_2, \cdots, x_n \in \R^n$. Then, for every $t>0$,
\begin{align*}
    \Pr_{ r \sim \{ \pm 1 \}^n }[ \| X \|_h > t ] \leq 2\exp \Big(-t^2 / \Big( 32\E_{ r \sim \{\pm 1\}^n }[\|X\|_h^2] \Big) \Big).
\end{align*}
\end{corollary}


By Theorem~\ref{thm:khinchin_kahane}, we know that any moments of $\|X\|_{h}$ are equivalent up to a constant factor. In particular,
\begin{claim}[Equivalence between first- and second-moment]\label{clm:hyperbolic_2_1_norm}
Given $n$ vectors $x_1, x_2, \cdots, x_n$. Let $r_1, r_2, \cdots, r_n$ denote a sequence of random Rademacher variables.
Let $X = \sum_{i=1}^n r_i x_i$. Then,
\begin{align*}
    ( \E[ \| X \|_h^2 ] )^{1/2}\leq \sqrt{2} \cdot \E[\|X\|_h].
\end{align*}
\end{claim}

We state two useful facts (Fact~\ref{fac:norm_h_q_bound_by_norm_h} and \ref{fac:norm_h_bound_by_norm_h_q}) that upper and lower bound the hyperbolic-$p$ norm by hyperbolic spectral norm.
\begin{fact}\label{fac:norm_h_q_bound_by_norm_h}
Let $h$ denote a $m$-variate degree-$d$ hyperbolic polynomial. For any vector $x$, for any $q>1$, we have
\begin{align*}
    \| x \|_{h,q} \leq d^{1/q} \cdot \| x \|_h.
\end{align*}
\end{fact}
\begin{proof}
We have
\begin{align*}
    \| x \|_{h,q} 
    =  \| \lambda(x) \|_{q} 
    \leq  d^{1/q} \cdot \| \lambda(x) \|_\infty 
    = d^{1/q} \cdot \| x \|_h.
\end{align*}
Thus, we complete the proof.
\end{proof}

\begin{fact}\label{fac:norm_h_bound_by_norm_h_q}
Let $h$ denote a $m$-variate degree-$d$ hyperbolic polynomial. For any vector $x$ and for any $q \geq 1$, we have
\begin{align*}
    \| x \|_{h} \leq \| x \|_{h,q}.
\end{align*}
\end{fact}
\begin{proof}
We have
\begin{align*}
    \| x \|_h = \| \lambda(x) \|_{\infty} \leq \| \lambda(x) \|_q = \| x \|_{h,q}.
\end{align*}
Thus, we complete the proof.
\end{proof}

\begin{fact}\label{fac:norm_h_matrix}
Let $h$ denote a $m$-variate degree-$d$ hyperbolic polynomial. For any vector $x$, if there exists a matrix $A\in \R^{d\times d}$ such that $\lambda(x) = \lambda(A)$, then we have
\begin{align*}
    \|x\|_h = \sigma_1(A).
\end{align*}
\end{fact}
\begin{proof}
We have
\begin{align*}
    \|x_1\|_h=\|\lambda(x_1)\|_\infty=\|\lambda(A_1)\|_\infty=\sigma_1(A_1).
\end{align*}
\end{proof}

We state a useful tool from previous work \cite{tj74,zyg02}.
\begin{lemma}[\cite{tj74,zyg02}]\label{lem:zyg02_tj74}
For $q\geq 2$, we have
\begin{align*}
    \binom{2q}{2k_1,\dots,2k_n} \leq M_{2q}^{2q} \cdot \binom{q}{k_1,\dots,k_n},
\end{align*}
where $M_{2q}= ( \frac{(2q)!}{2^q q!} )^{1/(2q)}$.
\end{lemma}

Using elementary calculations, we can upper bound $M_{2q}$.
\begin{fact}\label{fac:m_2q}
For any $q \geq 1$, we have
\begin{align*}
   \Big( \frac{ (2q)! }{2^q q! } \Big)^{1/(2q)} \leq \sqrt{2q - 1}.
\end{align*}
\end{fact}
\begin{proof}
We have
\begin{align*}
    \left( \frac{ (2q)! }{2^q q! } \right)^{1/(2q)} \leq &~ \left(\frac{e\cdot (2q)^{2q}\cdot \sqrt{2q}\cdot e^{-2q}}{2^q\cdot \sqrt{2\pi} \cdot q^{q} \cdot \sqrt{q}\cdot e^{-q}}\right)^{1/(2q)}\\
    = &~ \left(\frac{e^{1-q}}{\sqrt{\pi}}\cdot 2^q\cdot q^{q}\right)^{1/(2q)}\\
    \leq &~ \sqrt{2q-1},
\end{align*}
where the first step follows from Stirling's formula, and the last step follows from $q\geq 1$.
\end{proof}

\subsection{Proof of the Chernoff bound for hyperbolic polynomials}\label{sec:hyper_proof1}


The goal of this section is to prove Theorem~\ref{thm:main}.
\begin{theorem}[Chernoff bound for hyperbolic polynomial]\label{thm:main}
Let $h$ be an $m$-variable, degree-$s$ hyperbolic polynomial with respect to $e$. Given $x_1, x_2, \cdots, x_n \in \R^m$ such that $\rank(x_i)\leq s$ for all $i\in [n]$ and for some $0<s\leq d$. Let $\sigma = ( \sum_{i=1}^n \|x_i\|_h^2 )^{1/2}$. Then,
\begin{align*}
    \E_{ r \sim \{ \pm 1 \}^n }\left[\left\|\sum_{i=1}^n r_i x_i\right\|_h\right] \leq \min\{2\sqrt{\log s }, 1\}\cdot \sigma.
\end{align*}
Furthermore, there exist two constants $C_1,C_2>0$ such that for every $t>0$,
\begin{align*}
    \Pr_{ r \sim \{\pm 1\}^n } \left[ \left\| \sum_{i=1}^n r_i x_i \right\|_h > t \right]\leq C_1\exp\left(-\frac{C_2t^2}{\sigma^2 \log (s + 1)}\right).
\end{align*}
\end{theorem}

\begin{proof}
We first upper bound $\E_{ r \sim \{ \pm 1 \}^n }[\|\sum_{i=1}^n r_i x_i\|_h]$ by
\begin{align}\label{eq:bound_expect_h_norm}
    \E_{ r \sim \{ \pm 1 \}^n }\left[\left\|\sum_{i=1}^n r_i x_i\right\|_h\right] \leq &~  \E_{ r \sim \{ \pm 1 \}^n }\left[\left\|\sum_{i=1}^n r_i x_i\right\|_{h,2q}\right]\notag\\
    \leq & ~ \left( \E_{ r \sim \{ \pm 1 \}^n } \left[\left\|\sum_{i=1}^n r_i x_i\right\|_{h,2q}^{2q}\right] \right)^{1/(2q)}\\
    \leq & ~ \sqrt{2q-1} \cdot s^{1/(2q)} \cdot \left( \sum_{i=1}^n\|x_i\|_{h}^2 \right)^{1/2},\notag
\end{align}
where the first step follows from $\|x\|_{h}\leq \|x\|_{h,2q}$ when $q\geq 1$ (Fact~\ref{fac:norm_h_bound_by_norm_h_q}), the second step follows from the Lyapunov inequality (Lemma~\ref{lem:lyapunov}), and the third step follows from Lemma~\ref{lem:p_norm_upper_bound}.

Let's first assume $s>1$. By taking $q=\log s$, we have 
\begin{align*}
    \E_{ r \sim \{ \pm 1 \}^n }\left[\left\|\sum_{i=1}^n r_i x_i\right\|_h\right] 
    \leq & ~ \sqrt{ 4 (\log s ) - 2 }\cdot \left( \sum_{i=1}^n \|x_i\|_h^2 \right)^{1/2} \\
    = & ~ \sqrt{ 4 (\log s ) - 2 }\cdot \sigma \\ 
    \leq & ~ 2 \sqrt{\log s} \cdot \sigma.
\end{align*}
where the second step follows from $ \sigma := \left( \sum_{i=1}^n \|x_i\|_h^2 \right)^{1/2}$. When $s=1$, by taking $q=1$, we get that 
\begin{align*}
    \E_{ r \sim \{ \pm 1 \}^n }\left[\left\|\sum_{i=1}^n r_i x_i\right\|_h\right] 
    \leq ~ \sigma.
\end{align*}

By Claim~\ref{clm:hyperbolic_2_1_norm}, 
\begin{align*}
    \E_{ r \sim \{ \pm 1 \}^n }\left[\left\|\sum_{i=1}^n r_i x_i\right\|_h^2\right] 
    \leq 2\left( \E_{ r \sim \{ \pm 1 \}^n } \left[\left\|\sum_{i=1}^n r_i x_i\right\|_h\right] \right)^2 \leq \min\{8  \log s ,2\} \cdot \sigma^2.
\end{align*}
Then, by Corollary~\ref{cor:hyperbolic_prob},
\begin{align*}
    \Pr_{ r \sim \{ \pm 1 \}^n } \left[ \left\| \sum_{i=1}^n r_i x_i\right\|_h > t \right] \leq &~ 2 \exp \left(-\frac{t^2}{32 \E_{ r \sim \{ \pm 1 \}^n }[ \| \sum_{i=1}^n r_i x_i \|_h^2 ] }\right)\\
    \leq &~ 2\exp\left(-\frac{t^2}{64\min\{4  \log s ,1\} \cdot \sigma^2}\right).
\end{align*}
Thus, we complete the proof.
\end{proof}

\subsection{Expected hyperbolic-\texorpdfstring{$2q$}{2q} norm bound}\label{sec:exp_hyper_2p_norm}
The goal of this section is to prove Lemma~\ref{lem:p_norm_upper_bound}.
\begin{lemma}[Expected hyperbolic-$2q$ norm of Rademacher sum]\label{lem:p_norm_upper_bound}
Let $h$ be an $m$-variate, degree-$d$ hyperbolic polynomial. Given $n$ vectors $x_1, \cdots, x_n \in \R^m$ such that $\rank(x_i)\leq s$ for all $i\in [n]$ and for some $0<s\leq d$. For any $q\geq 1$, we have
\begin{align*}
   \left( \E_{r \sim \{\pm 1\}^n }\left[\left\|\sum_{i=1}^n r_i x_i\right\|_{h,2q}^{2q}\right] \right)^{1/(2q)}\leq \sqrt{2q-1} \cdot s^{1/(2q)}\cdot \left( \sum_{i=1}^n \|x_i\|_{h}^2 \right)^{1/2}.
\end{align*}
\end{lemma}

\begin{proof}

The main idea is to consider the random variables $r_1, r_2, \cdots, r_n$ one at a time. By the conditional expectation, we have
\begin{align*}
    \E_{ r \sim \{ \pm 1 \}^n } \left[\left\|\sum_{i=1}^n r_i x_i\right\|_{h,2q}^{2q}\right] 
    = & ~ \E_{ r_2, \cdots, r_n \sim \{ \pm 1 \} } \left[ \E_{ r_1 \sim \{ \pm 1\} } \left[\left\|\sum_{i=1}^n r_i x_i\right\|_{h,2q}^{2q}\right] \right] \\
    = & ~ \E_{r_2,\dots,r_n \sim \{\pm 1\} } \left[ \E_{r_1 \sim \{ \pm 1\} } \left[ \sum_{j=1}^d \lambda_j\left(r_1 x_1 + \sum_{i=2}^n r_ix_i\right)^{2q} \right] \right].
\end{align*}
where the last step follows from the definition of $\|\cdot\|_{h,2q}$ norm.

To apply Corollary~\ref{cor:hv_2_vars}, let $x=x_1,y=\sum_{i=2}^n r_i x_i$. Then, there exists two symmetric matrices $A_1,B_1\in \R^{d\times d}$ such that 
\begin{align}\label{eq:hyperbolic_matrix_eigen}
    \lambda\left( r_1 x_1 + \sum_{i=2}^n r_ix_i \right) = \lambda(r_1 A_1+B_1),
\end{align}
where $\lambda$ is the vector of eigenvalues ordered from large to small. Then, we have
\begin{align*}
    \lambda(x_1) = \lambda(A_1),\quad \lambda\left(\sum_{i=2}^n r_i x_i\right) = \lambda(B_1).
\end{align*}

Hence, by the definition of Schatten-$p$ norm,
\begin{align}\label{eq:eigen_trace}
    \sum_{j=1}^d \left( \lambda_j \Big(r_1 x_1 + \sum_{i=2}^n r_ix_i \Big) \right)^{2q} =&~  \|r_1 A_1 +B_1\|_{2q}^{2q}\notag\\
    = &~ \tr\left[(r_1A_1+B_1)^{2q}\right]\notag\\
    = &~ \sum_{ \beta \in \{0,1\}^{2q} } \tr\left[\prod_{i=1}^{2q}A_1^{\beta_i}B_1^{1-\beta_i}\right]\cdot r_1^{\sum_{i=1}^{2q} \beta_i}.
\end{align}
where the first step follows from Eq.~\eqref{eq:hyperbolic_matrix_eigen} and the definition of matrix Schatten $p$-norm, the second step follows from $\|A\|_{2q}^{2q}=\tr[A^{2q}]$ for symmetric matrix $A$ and $q\geq 1$, and the last step follows from the linearity of trace.

We define a set which will be used later.
\begin{align*}
     \mathcal{B}_{\mathrm{even}} :=\left\{\beta\in \{0,1\}^{2q}: \sum_{i=1}^{2q} \beta_i \text{~is even}\right\}.
\end{align*}

By taking expectation for $r_1$, we have
\begin{align*}
    \E_{ r_1 \sim \{ \pm 1 \} }\left[ \sum_{j=1}^d \lambda_j\left(r_1 x_1 + \sum_{i=2}^n r_ix_i\right)^{2q} \right] 
    = &~ \sum_{\beta\in \{0,1\}^{2q}} \tr\left[\prod_{i=1}^{2q}A_1^{\beta_i}B_1^{1-\beta_i}\right]\cdot \E_{r_1 \sim \{\pm 1\} }\left[r_1^{\sum_{i=1}^{2q} \beta_i}\right]\\
    = &~ \sum_{\beta\in \mathcal{B}_{\mathrm{even}}} \tr\left[\prod_{i=1}^{2q}A_1^{\beta_i}B_1^{1-\beta_i}\right]
\end{align*}
where the first step follows from Eq.~\eqref{eq:eigen_trace} and the linearity of expectation, and the last step follows from
\begin{align*}
    \E_{r_1 \sim \{\pm 1\} }\big[r_1^{k}\big] = \begin{cases}
    0 & \text{if }k\text{~is odd},\\
    1 & \text{if }k\text{~is even}.
    \end{cases}
\end{align*}

For each $\beta\in \mathcal{B}_{\mathrm{even}}$, we have
\begin{align*}
    \tr\left[\prod_{i=1}^{2q}A_1^{\beta_i}B_1^{1-\beta_i}\right] \leq \sum_{j=1}^s \sigma_j\left(\prod_{i=1}^{2q}A_1^{\beta_i}B_1^{1-\beta_i}\right)\leq \sum_{j=1}^s \prod_{i=1}^{2q}\sigma_j\left(A_1^{\beta_i}B_1^{1-\beta_i}\right),
\end{align*}
where $\sigma_j(A)$ is the $j$-th singular value of $A$ and the first step follows from $\tr[A]\leq \sum_{i=1}^{\rank(A)}\sigma_j(A)$
for any real square matrix $A$, and the second step follows from general Horn inequality (Lemma~\ref{lem:horn}).

Then, it follows that
\begin{align}\label{eq:trace_to_singular}
    \sum_{\beta\in \mathcal{B}_{\mathrm{even}}} \tr\left[\prod_{i=1}^{2q}A_1^{\beta_i}B_1^{1-\beta_i}\right] \leq &~ 
    \sum_{\beta\in \mathcal{B}_{\mathrm{even}}}\sum_{j=1}^s \sigma_j(A_1)^{\sum_{i=1}^{2q}\beta_i} \sigma_j(B_1)^{2q-\sum_{i=1}^{2q}\beta_i}\notag\\
    = &~ \sum_{k=0}^q \binom{2q}{2k} \sum_{j=1}^s \sigma_j(A_1)^{2k}\sigma_j(B_1)^{2q-2k},
\end{align}
where the first step follows from $\rank(A)=\rank(x_1)\leq s$.
Hence,
\begin{align}\label{eq:iter}
    \E_{r_1,\dots,r_n \sim \{ \pm 1\}}\left[\left\|\sum_{i=1}^n r_i x_i\right\|_{h,2q}^{2q}\right] \leq &~ \sum_{k_1=0}^q \binom{2q}{2k_1}\E_{r_2,\dots,r_n}\left[ \sum_{j=1}^s \sigma_j(A_1)^{2k_1}\sigma_j(B_1)^{2q-2k_1} \right]\notag\\
    \leq &~ \sum_{k_1=0}^q \binom{2q}{2k_1}\E_{r_2,\dots,r_n}\left[ \sigma_1(A_1)^{2k_1} \sum_{j=1}^s \sigma_j(B_1)^{2q-2k_1} \right]\notag\\
    = &~ \sum_{k_1=0}^q \binom{2q}{2k_1} \E_{r_2,\dots,r_n}\left[\|x_1\|_h^{2k_1}\sum_{j=1}^s \lambda_j(B_1)^{2q-2k_1} \right]\notag\\
    = &~ \sum_{k_1=0}^q \binom{2q}{2k_1}\|x_1\|_h^{2k_1} \E_{r_2,\dots,r_n}\left[\sum_{j=1}^s \lambda_j(B_1)^{2q-2k_1} \right]\notag \\
    = &~ \sum_{k_1=0}^q \binom{2q}{2k_1}\|x_1\|_h^{2k_1} \E_{r_2,\dots,r_n}\left[\left\|\sum_{i=2}^n r_i x_i \right\|_{h,2q-2k_1}^{2q-2k_1} \right],
\end{align}
where the second step follows from $\sigma_1(A)\geq \cdots \geq \sigma_s(A)$, the third step follows from for $\sum_{i=1}^s \sigma_i(A)^k=\sum_{i=1}^s\lambda_i(A)^k$ for even $k$, the forth step follows from Fact~\ref{fac:norm_h_matrix}, and the last step follows from definition of $\| \cdot \|_{h,q}$.

Now, we can iterate this process for $\E_{r_2,\dots,r_n}\left[\left\|\sum_{i=2}^n r_i x_i \right\|_{h,2q-2k_1}^{2q-2k_1} \right]$. Consider $r_2 x_2 + \sum_{i=3}^{n} r_i x_i$. By Corollary~\ref{cor:hv_2_vars}, there exists two symmetric matrices $A_2,B_2\in \R^{d\times d}$ such that
\begin{align*}
    \lambda\left(r_2 x_2 + \sum_{i=3}^{n} r_i x_i\right) = \lambda(r_2 A + B)
\end{align*}
for all $r_2\in \{-1, 1\}$. By the conditional expectation again, we can get that
\begin{align*}
    \E_{r_1,\dots,r_n}\left[\left\|\sum_{i=1}^n r_i x_i\right\|_{h,2q}^{2q}\right] \leq &~  \sum_{k_1=0}^q \binom{2q}{2k_1}\|x_1\|_h^{2k_1} \E_{r_2,\dots,r_n}\left[\left\|\sum_{i=2}^n r_i x_i \right\|_{h,2q-2k_1}^{2q-2k_1} \right] \\
    \leq &~ \sum_{k_1=0}^q \binom{2q}{2k_1}\|x_1\|_h^{2k_1} \sum_{k_2=0}^{2q-2k_1} \binom{2q-2k_1}{2k_2} \|x_2\|_h^{2k_2}\E_{r_3,\dots,r_n}\left[\left\|\sum_{i=3}^n r_i x_i \right\|_{h,2k_3}^{2k_3} \right],
\end{align*}
where $k_3=q-k_1-k_2$ and the second step follows from applying Eq.~\eqref{eq:iter} for $\E_{r_2,\dots,r_n}\left[\left\|\sum_{i=2}^n r_i x_i \right\|_{h,2q-2k_1}^{2q-2k_1} \right]$.

If we iterate $n-1$ times, we finally get
\begin{align}\label{eq:expansion_2q_norm}
    \E\left[\left\|\sum_{i=1}^n r_i x_i\right\|_{h,2q}^{2q}\right] 
    \leq & ~ \sum_{\substack{k_1,\dots,k_n\geq 0\\k_1+\cdots+k_n=q}} \prod_{i=1}^{n-1} \binom{2q - \sum_{j=1}^{i-1} 2k_j}{2k_{i}}  \|x_i\|_h^{2k_i} \cdot \E_{r_n}\left[ \|r_n x_n\|_{h,2k_n}^{2k_n} \right] \notag \\
    = &~ \sum_{\substack{k_1,\dots,k_n\geq 0\\k_1+\cdots+k_n=q}}\binom{2q}{2k_1,\dots,2k_n} \prod_{i=1}^{n-1} \|x_i\|_h^{2k_i} \cdot \E_{r_n}\left[ \|r_n x_n\|_{h,2k_n}^{2k_n} \right]\notag\\
    = &~ \sum_{\substack{k_1,\dots,k_n\geq 0\\k_1+\cdots+k_n=q}}\binom{2q}{2k_1,\dots,2k_n} \prod_{i=1}^{n-1} \|x_i\|_h^{2k_i} \cdot \|x_n\|_{h,2k_n}^{2k_n},
\end{align}
where the first step follows from iterating the same rule for $n-1$ times, the second step follows from
\begin{align*}
    \prod_{i=1}^{n-1} \binom{2q - \sum_{j=1}^{i-1}2k_{j}}{2k_{i}} = &~  \binom{2q}{2k_1}\cdot \binom{2q-2k_1}{2k_2}\cdot \binom{2q-2k_1-2k_2}{2k_3}\cdots \binom{2k_{n-1} + 2k_n}{2k_{n-1}}\\
    = &~ \binom{2q}{2k_1}\cdot \binom{2q-2k_1}{2k_2}\cdot \binom{2q-2k_1-2k_2}{2k_3}\cdots \binom{2k_{n-1} + 2k_n}{2k_{n-1}} \cdot \binom{ 2k_n}{2k_{n}}\\
    = &~ \binom{2q}{2k_1,\dots,2k_n}.
\end{align*}

By Lemma~\ref{lem:zyg02_tj74}, we have that
\begin{align*}
    \binom{2q}{2k_1,\dots,2k_n}\leq M_{2q}^{2q} \cdot \binom{q}{k_1,\dots,k_n},
\end{align*}
where $M_{2q}=\left(\frac{(2q)!}{2^q q!}\right)^{1/(2q)}\leq \sqrt{2q-1}$ by Fact~\ref{fac:m_2q}.

Hence,
\begin{align*}
    \E_{ r \sim \{\pm 1\}^n }\left[\left\|\sum_{i=1}^n r_i x_i\right\|_{h,2q}^{2q}\right] \leq &~  M_{2q}^{2q}\sum_{\substack{k_1,\dots,k_n\geq 0\\k_1+\cdots+k_n=q}}\binom{q}{k_1,\dots,k_n} \prod_{i=1}^{n-1} \|x_i\|_h^{2k_i} \cdot \|x_n\|_{h,2k_n}^{2k_n}\\
    \leq &~  M_{2q}^{2q}\sum_{\substack{k_1,\dots,k_n\geq 0\\k_1+\cdots+k_n=q}}\binom{q}{k_1,\dots,k_n} \prod_{i=1}^{n-1} \|x_i\|_h^{2k_i} \cdot s\cdot \|x_n\|_{h}^{2k_n}\\
    = & ~ M_{2q}^{2q} \cdot s \cdot \sum_{\substack{k_1,\dots,k_n\geq 0\\k_1+\cdots+k_n=q}}\binom{q}{k_1,\dots,k_n} \prod_{i=1}^{n} \|x_i\|_h^{2k_i}  \\
    = &~ M_{2q}^{2q}\cdot s \cdot \left(\sum_{i=1}^{n} \|x_i\|_h^2 \right)^q,
\end{align*}
where the second step follows from $\|x_n\|_{h,2k_n} \leq s^{1/(2k_n)}\cdot \|x_n\|_h$ for rank-$s$ vector (see Fact~\ref{fac:norm_h_q_bound_by_norm_h}), the third step follows from re-organizing the terms, and the last step follows from expanding $(\sum_{i=1}^{n} \|x_i\|_h^2 )^q$.

Therefore,
\begin{align*}
    \left(\E_{r \sim \{\pm 1\}^n}\left[\left\|\sum_{i=1}^n r_i x_i\right\|_{h,2q}^{2q}\right]\right)^{1/(2q)} \leq \sqrt{2q-1}\cdot s^{1/(2q)}\cdot \left(\sum_{i=1}^{n} \|x_i\|_h^2\right)^{1/2},
\end{align*}
which completes the proof of Lemma~\ref{lem:p_norm_upper_bound}.

\begin{remark}
This upper bound depends essentially on the the maximum hyperbolic rank of all the vectors $x_1,\dots,x_n$, instead of just the last one. It follows since Eq.~\eqref{eq:iter} can be expanded as:
\begin{align*}
    \E_{r_1,\dots,r_n \sim \{ \pm 1\}}\left[\left\|\sum_{i=1}^n r_i x_i\right\|_{h,2q}^{2q}\right] \leq &~ \E_{r_2,\dots,r_n\sim \{\pm 1\}}\left[\left\|\sum_{i=2}^n r_i x_i\right\|_{h,2q}^{2q}\right] + \|x_1\|_{h,2q}^{2q}\\
    &+\sum_{k_1=1}^{q-1} \binom{2q}{2k_1} \|x_1\|_{h}^{2k_1}\cdot \E_{r_2,\dots,r_n\sim \{\pm 1\}}\left[\left\|\sum_{i=2}^n r_i x_i\right\|_{h,2q-k_1}^{2q-k_1}\right]. 
\end{align*}
We can see that the first term depends on the rank of $x_2,\dots,x_n$, the second term depends on the $\rank(x_1)$. Hence, the whole summation cannot be uniformly bounded by the rank of the last vector $x_n$. Hence, adding a rank-1 dummy vector cannot improve the bound.
\end{remark}

\end{proof}

\section{Hyperbolic Chernoff bound for hyperbolic cone vectors}\label{sec:hyper_proof2}
The goal of this section is to prove the following theorem, which generalizes the matrix Chernoff bound for positive semi-definite matrices to the hyperbolic version with respect to random vectors in the hyperbolic cone.
\begin{theorem}\label{thm:chernoff_cone}
Let $h$ be an $m$-variate, degree-$d$ hyperbolic polynomial with hyperbolic direction $e\in \R^m$. Let $\Lambda_+$ denote the hyperbolic cone of $h$ with respect to $e$. Suppose $\mathsf{x}_1,\dots,\mathsf{x}_n$ are $n$ independent random vectors with supports in $\Lambda_+$ such that $\lambda_{\max}(\mathsf{x}_i)\leq R$ for all $i\in [n]$.

Define the mean of minimum and maximum eigenvalues as follows:
\begin{align*}
    \mu_{\min}:=\sum_{i=1}^n \E[\lambda_{\min}(\mathsf{x}_i)], ~~\text{and}~~\mu_{\max}:=\sum_{i=1}^n \E[\lambda_{\max}(\mathsf{x}_i)].
\end{align*}

Then, we have
\begin{align*}
    \Pr\left[\lambda_{\max}\left(\sum_{i=1}^n \mathsf{x}_i\right)\geq (1+\delta)\mu_{\max}\right] \leq&~ d\cdot \left(\frac{(1+\delta)^{1+\delta}}{e^{\delta}}\right)^{-\mu_{\max}/R}~~\forall \delta\geq 0,\\
    \Pr\left[\lambda_{\min}\left(\sum_{i=1}^n \mathsf{x}_i\right)\leq (1-\delta)\mu_{\min}\right] \leq&~ d\cdot \left(\frac{(1-\delta)^{1-\delta}}{e^{-\delta}}\right)^{-\mu_{\min}/R}~~\forall \delta\in [0,1].
\end{align*}
\end{theorem}

\begin{proof}
Without loss of generality, we may assume that $\lambda_{\max}(\mathsf{x}_i)\leq 1$. The general case will follow from scaling.

\paragraph*{Maximum eigenvalue: } By the Laplace transform method, we have
\begin{align}\label{eq:exp_sum}
    \Pr\left[\lambda_{\max}\left(\sum_{i=1}^n \mathsf{x}_i\right)\geq t\right] \leq &~  \inf_{\theta >0}~e^{-\theta t}\cdot \E\left[\exp\left(\theta\lambda_{\max}\left(\sum_{i=1}^n \mathsf{x}_i\right)\right)\right]\notag\\
    = &~ \inf_{\theta >0}~e^{-\theta t}\cdot \E\left[\sum_{q\geq 0}\frac{\theta^q}{q!}\lambda_{\max}\left(\sum_{i=1}^n \mathsf{x}_i\right)^q\right]\notag\\
    \leq &~ \inf_{\theta >0}~e^{-\theta t}\cdot\sum_{q\geq 0}\frac{\theta^q}{q!}\E\left[\sum_{j=1}^d\lambda_{j}\left(\sum_{i=1}^n \mathsf{x}_i\right)^q\right],
\end{align}
where the second step follows from Taylor expansion, and the third step follows from $\mathsf{x}_i\in \Lambda_+$ and each term in the summation are non-negative.

Then, the remaining task is very similar to the proof of Lemma~\ref{lem:p_norm_upper_bound}. We will upper bound the expectation of trace moments as follows: let $\E_i$ denote the expectation over $\mathsf{x}_i$, and $\E_{\geq i}$ denote the expectation over $\mathsf{x}_{i},\dots,\mathsf{x}_{n}$. Then, we have
\begin{align*}
    \mathbb{E}_{\geq 1}\left[\sum_{j=1}^d\lambda_{j}\left(\sum_{i=1}^n \mathsf{x}_i\right)^q\right]= &~ \mathbb{E}_{\geq 2} \mathbb{E}_1\left[\sum_{j=1}^d\lambda_{j}\left(\mathsf{x}_1 + \sum_{i=2}^n \mathsf{x}_i\right)^q\right]\\
    = &~ \mathbb{E}_{\geq 2} \mathbb{E}_1\left[\tr\left[(A_1+B_1)^q\right]\right],
\end{align*}
where $A_1, B_1\in \R^{d\times d}$ are two symmetric matrices given by Corollary~\ref{cor:hv_2_vars} such that $A_1$ depends on the value of $\mathsf{x}_1$ and $B_1$ depends on the values of $\mathsf{x}_2,\dots,\mathsf{x}_n$. Since all eigenvalues of $\mathsf{x}_1,\dots, \mathsf{x}_n$ are non-negative, $A$ and $B$ are positive semi-definite matrices. Then, we have
\begin{align*}
    \mathbb{E}_{\geq 2} \mathbb{E}_1\left[\tr\left[(A_1+B_1)^q\right]\right] = &~ \mathbb{E}_{\geq 2} \mathbb{E}_1\left[\sum_{\beta\in \{0,1\}^{q}}\tr\left[\prod_{k=1}^q A_1^{\beta_i} B_1^{1-\beta_i}\right]\right]\\
    \leq &~ \mathbb{E}_{\geq 2} \mathbb{E}_1\left[\sum_{\beta\in \{0,1\}^{q}}\sum_{j=1}^d \lambda_j(A_1)^{\sum_{i=1}^q \beta_i}\cdot \lambda_j(B_1)^{q - \sum_{i=1}^q \beta_i}\right]\\
    = &~ \mathbb{E}_{\geq 2} \mathbb{E}_1\left[\sum_{k_1=0}^q\binom{q}{k_1}\sum_{j=1}^d \lambda_j(A_1)^{k_1}\cdot \lambda_j(B_1)^{q - k_1}\right]\\
    \leq &~ \mathbb{E}_{\geq 2} \mathbb{E}_1\left[\sum_{k_1=0}^q\binom{q}{k_1}\lambda_{\max}(A_1)^{k_1}\cdot \sum_{j=1}^d \lambda_j(B)^{q - k_1}\right]\\
    = &~ \mathbb{E}_{1} \left[\sum_{k_1=0}^q\binom{q}{k_1}\lambda_{\max}(\mathsf{x}_1)^{k_1}\cdot \mathbb{E}_{\geq 2} \left[\sum_{j=1}^d \lambda_j\left(\sum_{i=2}^n \mathsf{x}_i\right)^{q - k_1}\right]\right].
\end{align*}
where the second step follows from the repeated application of Horn's inequality (Lemma~\ref{lem:horn}) and $A_1, B$ are positive semi-matrices, and the last step follows from $\mathsf{x}_1$ is independent with $\mathsf{x}_2,\dots,\mathsf{x}_n$.

Then, by repeating this process, we finally get that
\begin{align*}
    \mathbb{E}\left[\sum_{j=1}^d\lambda_{j}\left(\sum_{i=1}^n \mathsf{x}_i\right)^q\right]\leq &~ \E\left[ \sum_{\substack{k_1,\dots,k_n\geq 0\\k_1+\cdots+k_n=q}} \binom{q}{k_1,\dots,k_n} \prod_{i=1}^n \lambda_{\max}(\mathsf{x}_i)^{k_i}\cdot d \right]\\
    = &~ d\cdot \E\left[\left(\sum_{i=1}^n \lambda_{\max}(\mathsf{x}_i)\right)^q\right].
\end{align*}
Putting it to Eq.~\eqref{eq:exp_sum}, we have
\begin{align*}
    \Pr\left[\lambda_{\max}\left(\sum_{i=1}^n \mathsf{x}_i\right)\geq t\right] 
    \leq &~ \inf_{\theta >0}~e^{-\theta t}\cdot\sum_{q\geq 0}\frac{\theta^q}{q!}\E\left[\sum_{j=1}^d\lambda_{j}\left(\sum_{i=1}^n \mathsf{x}_i\right)^q\right]\\
    \leq &~ \inf_{\theta >0}~e^{-\theta t}\cdot\sum_{q\geq 0}\frac{\theta^q}{q!}\cdot d\cdot \E\left[\left(\sum_{i=1}^n \lambda_{\max}(\mathsf{x}_i)\right)^q\right]\\
    = &~ \inf_{\theta >0}~e^{-\theta t}\cdot d\cdot \E\left[\exp\left(\theta \cdot \sum_{i=1}^n \lambda_{\max}(\mathsf{x}_i)\right)\right]\\
    = &~ \inf_{\theta >0}~e^{-\theta t}\cdot d\cdot \prod_{i=1}^n \E\left[e^{\theta\lambda_{\max}(\mathsf{x}_i)}\right],
\end{align*}
where the third step follows from the linearity of expectation, and the last step follows from the independence of $\mathsf{x}_1,\dots,\mathsf{x}_n$.

For $x\in [0, 1]$, we know that $e^{\theta x}\leq 1 + (e^\theta-1)x$ holds for $\theta \in \R$. Thus,
\begin{align*}
    e^{-\theta t}\cdot d\cdot \prod_{i=1}^n \E\left[e^{\theta\lambda_{\max}(\mathsf{x}_i)}\right] \leq &~ e^{-\theta t}\cdot d\cdot \prod_{i=1}^n (1 + (e^\theta - 1)\E[\lambda_{\max}(\mathsf{x}_i)])\\
    = &~ d\cdot \exp\left(-\theta t + \sum_{i=1}^n\log\left(1 + (e^\theta - 1)\E[\lambda_{\max}(\mathsf{x}_i)]\right)\right)\\
    \leq &~ d\cdot \exp\left(-\theta t + \sum_{i=1}^n(e^\theta - 1)\E[\lambda_{\max}(\mathsf{x}_i)]\right)\\
    = &~ d\cdot \exp\left(-\theta t + (e^\theta - 1)\mu_{\max}\right)
\end{align*}
where the second step follows from our assumption that $\lambda_{\max}(x)\in [0, 1]$ for all $i\in [n]$, and the third step follows from $\log(1+x)\leq x$ for $x>-1$. Therefore, by taking $\theta := \log(t/\mu_{\max})$, we have
\begin{align}\label{eq:concen_geq_t}
    \Pr\left[\lambda_{\max}\left(\sum_{i=1}^n \mathsf{x}_i\right)\geq t\right] 
    \leq d\cdot \left(\frac{t}{\mu_{\max}}\right)^{-t}\cdot e^{t-\mu}.
\end{align}
If we choose $t:=(1+\delta)\mu_{\max}$, we get that
\begin{align*}
    \Pr\left[\lambda_{\max}\left(\sum_{i=1}^n \mathsf{x}_i\right)\geq (1+\delta)\mu_{\max}\right] \leq&~ d\cdot \left(\frac{(1+\delta)^{1+\delta}}{e^{\delta}}\right)^{-\mu_{\max}},
\end{align*}
which completes the proof of the maximum eigenvalue case.

\paragraph*{Minimum eigenvalue: }
We reduce this case to the maximum eigenvalue case by defining $\mathsf{x}_i':= e-\mathsf{x}_i$ for $i\in [n]$. Then, by Fact~\ref{fac:eigen_linear_trans}, 
\begin{align*}
    \lambda_{\max}(\mathsf{x}_i')= 1-\lambda_{\min}(\mathsf{x}_i)\leq 1, ~~\text{and}~~\lambda_{\min}(\mathsf{x}_i')= 1-\lambda_{\max}(\mathsf{x}_i)\geq 0.
\end{align*}
Thus,
\begin{align*}
    \Pr\left[\lambda_{\min}\left(\sum_{i=1}^n \mathsf{x}_i\right)\leq (1-\delta)\mu_{\min}\right] = &~ \Pr\left[\lambda_{\max}\left(\sum_{i=1}^n \mathsf{x}_i'\right)\geq n - (1-\delta)\mu_{\min}\right]\\
    \leq &~ d\cdot \left(\frac{n - (1-\delta)\mu_{\min}}{n - \mu_{\min}}\right)^{n - (1-\delta)\mu_{\min}}\cdot e^{\delta \mu_{\min}}\\
    = &~ d\cdot \left(1 + \frac{\delta}{n/\mu_{\min} - 1}\right)^{\left(\frac{n}{(1-\delta)\mu_{\min}} - 1\right) \cdot (1-\delta)\mu_{\min}}\cdot e^{\delta \mu_{\min}}\\
    \leq &~ d\cdot \left(\frac{(1-\delta)^{1-\delta}}{e^{-\delta}}\right)^{-\mu_{min}},
\end{align*}
where the second step follows from taking $t:=n-(1-\delta)\mu_{\min}$ in Eq.~\eqref{eq:concen_geq_t} and $\mu_{\max}' =n-\mu_{\min}$, the last step follows from $n/\mu_{\min}>0$.

Hence, the proof of the theorem is completed.
\end{proof}
\section{Hyperbolic anti-concentration bound}\label{sec:anticon}

\subsection{Our result}
In this section, we will prove an anti-concentration bound for random vectors with respect to the hyperbolic norm, which generalizes the result for PSD matrices in \cite{ay21}. In particular, an important tool we use is the hyperbolic Chernoff bound for random vectors in the hyperbolic cone (Theorem~\ref{thm:chernoff_cone}), together with a robust Littlewood-Offord theorem for hyperbolic cone (Theorem~\ref{thm:anticon_good}).

\begin{theorem}[Hyperbolic anti-concentration theorem]\label{thm:anti-concen}
Let $h_1,h_2$ be an $m$-variate degree-$d$ hyperbolic polynomial with hyperbolic direction $e_1,e_2\in \R^m$, respectively. Let $y_1, y_2 \in \R^m$ be two vectors. Let $\{x^1_i\}_{i\in [n]}$ and $\{x^2_i\}_{i\in [n]}$ be two sequences of vectors such that $x_i^1\in \Lambda_{+,h_1}$ and $x_i^2\in (-\Lambda_{+,h_2})$, i.e., $\lambda_{\min,h_1}(x_i^{1})\geq 0$, $\lambda_{\max,h_2}(x_i^{2})\leq 0$ for all $i\in [n]$. 

Let $\tau \leq \frac{1}{\sqrt{\log d}}$. We further assume that 
$\lambda_{\max, h_1}(x^1_i)\leq \tau$ and $\lambda_{\min,h_2}(x^{2}_i)\geq -\tau$ for all $i\in [n]$. And for $j\in [2]$, we have $\sum_{i=1}^n \lambda_{\min}(x_i^j)^2\geq 1$.

Then, for $\Delta \geq 20\tau \log d$, we have
\begin{align*}
    \Pr_{\epsilon\sim \{-1,1\}^n}\left[\exists j\in [2]: ~\left\|\sum_{i=1}^n \epsilon_i x^j_i-y_j\right\|_{h_j} \leq \Delta \right]\leq O(\Delta).
\end{align*}
\end{theorem}
\begin{proof}
We follow the proof in \cite{ay21} but adapt it to the hyperbolic polynomial.

Let $f_j(\epsilon):=\sum_{i=1}^n \epsilon_ix_i^j$ for $j\in [2]$. And we will first show that
\begin{align*}
    \Pr_{\epsilon\sim \{\pm 1\}^n}[\exists j\in [2]: \lambda_{\max,h_j}(f_j(\epsilon)-y_j)\leq \Delta]\leq O(\Delta),
\end{align*}
which implies the anti-concentration bound for the hyperbolic spectral norm.

Let $p:=\frac{1}{20\tau^2\log d}$ and let $\pi: [n]\rightarrow [2p]$ be a random hash function that independently assigns each $i \in [n]$ to uniformly random bucket in $[2p]$. For $i\in [2p]$, let $C_i:=\{j\in [n]: \pi[j]=i\}$ be the set of elements in the $i$-th bucket. Let $\gamma\sim \{\pm 1\}^{2p}$. For $j\in [2]$, define a new function $g_j(\gamma):\{\pm\}^{2k}\rightarrow \R^m$ as follows:
\begin{align*}
    g_j(\gamma) := \sum_{i=1}^{2p} \gamma_i \cdot \sum_{j\in C_i} x_i^j.
\end{align*}
That is, we assign the same sign for vectors hashed into the same bucket.

\cite{ay21} proved that $f_j(\epsilon)$ and $g_j(\gamma)$ have the same distribution using a direct argument about the random hash function. Thus, it is also true in our case and we just need to prove
\begin{align*}
    \Pr_{\gamma\sim \{\pm 1\}^{2p}}[\exists j\in [2]: \lambda_{\max,h_j}(g_j(\gamma)-y_j)\leq \Delta]\leq O(\Delta)
\end{align*}

For $j=1$, define the good bucket set
\begin{align*}
     {\cal B}_{\mathrm{good}}^1:=\left\{c\in [2p]: \lambda_{\min,h_1}\left(\sum_{i\in \pi^{-1}(c)} x_i^1\right)\geq \frac{1}{4\tau p}\right\}.
\end{align*}
By Lemma~\ref{lem:bound_good_buckets}, with probability at least $1-e^{-p/2}$, we have $|{\cal B}_{\mathrm{good}}^1|\geq \frac{8}{5}p$.  

For $j=2$, define the good bucket set
\begin{align*}
    {\cal B}_{\mathrm{good}}^2:=\left\{c\in [2p]: \lambda_{\max,h_2}\left(\sum_{i\in \pi^{-1}(c)} x_i^2\right)\leq -\frac{1}{4\tau p}\right\}.
\end{align*}
By considering $-x_i^2$ and applying Lemma~\ref{lem:bound_good_buckets}, we get that with with probability at least $1-e^{-p/2}$, we have $|{\cal B}_{\mathrm{good}}^2|\geq \frac{8}{5}p$.

By a pigeonhole principle and union bound, with probability $1-2e^{-p/2}$, $|{\cal B}_{\mathrm{good}}^1 \cap {\cal B}_{\mathrm{good}}^2|\geq \frac{6}{5}p$. That is, at least $\frac{3}{5}$-fraction of $i\in [2p]$ such that 
\begin{align*}
    \lambda_{\min,h_1}\left(\sum_{j\in \pi^{-1}(c)}x_i^1\right)\geq \frac{1}{4\tau p}, ~~\text{and}~~ \lambda_{\max,h_2}\left(\sum_{j\in \pi^{-1}(c)}x_i^2\right)\leq -\frac{1}{4\tau p}.
\end{align*}
Thus, we can apply Theorem~\ref{thm:anticon_good} with $\alpha=\frac{3}{5}, \rho = \frac{1}{4\tau p}$ and get that
\begin{align*}
    \Pr_{\gamma\sim \{-1,1\}^{2p}}\left[ \exists j\in [2]:~\lambda_{\max,h_j}\left(\sum_{i=1}^{2p} \gamma_i \cdot \sum_{k\in \pi^{-1}}x_k^j-y_j\right)\in \left(-\frac{1}{2\tau p},0\right] \right]\leq O\left(\frac{1}{\sqrt{p}}\right)+2e^{-p/2}.
\end{align*}
Now, we transform back to the distribution of $f_j(\epsilon)$ and have the following bound:
\begin{align}\label{eq:bound_small_interval}
    \Pr_{\epsilon\sim \{\pm 1\}^n}\left[\exists j\in [2]: \lambda_{\max,h_j}(f_j(\epsilon)-y_j)\in \left(-\frac{1}{2\tau p},0\right]\right]\leq O\left(\frac{1}{\sqrt{p}}\right)+2e^{-p/2}.
\end{align}
However, by our choice of parameters, $\Delta \geq \frac{1}{2\tau p}$. We can partition the interval $[-\Delta, 0]$ into $\lceil 2\tau p\Delta\rceil$ sub-intervals each of length $\frac{1}{2\tau p}$. Since Eq.~\eqref{eq:bound_small_interval} holds for any $y_j\in \R^m$, we can use it to bound the probability of the event $\lambda_{\max,h_j}(f_j(\epsilon)-y_j)\in (-\frac{k}{2\tau p}, -\frac{k-1}{2\tau p}]$ by shifting $y_j':=y_j + \frac{k-1}{2\tau p}\cdot e$. Therefore, by the union bound, we have
\begin{align*}
    \Pr_{\epsilon\sim \{\pm 1\}^n}\left[\exists j\in [2]: \lambda_{\max,h_j}(f_j(\epsilon)-y_j)\in [\Delta,0]\right]\leq &~ \lceil 2\tau p\Delta\rceil\cdot \left(O\left(\frac{1}{\sqrt{p}}\right)+2e^{-p/2}\right)\\
    = &~ O(\Delta \cdot \tau \sqrt{p})\\
    = &~ O(\Delta),
\end{align*}
where the last step follows from $\tau \sqrt{p} = O(1)$ by our choice of parameters.

We note that the above upper bound also holds for the interval $[0, \Delta]$. Hence, we complete the proof of the theorem.
\end{proof}

\subsection{Technical lemmas}
To prove Theorem~\ref{thm:anti-concen}, we need a robust Littlewood–Offord theorem for hyperbolic cones. This kind of theorems were previously proved by \cite{ost19} for polytopes and \cite{ay21} for positive spectrahedrons.

We first give some definitions in \cite{ost19} about functions on hypercube.
\begin{definition}[Unateness]
A function $F : \{-1, 1\}^n\rightarrow \{0, 1\}$ is unate if for all $i\in [n]$, $F$ is either increasing or decreasing with respect to the $i$th coordinate, i.e.,
\begin{align*}
    F(x_1,\dots,x_{i-1}, -1, x_{i+1},\dots,x_n) \leq&~ F(x_1,\dots,x_{i-1}, 1, x_{i+1},\dots,x_n)~~~\forall x\in \{\pm 1\}^n, ~~\text{or}\\
    F(x_1,\dots,x_{i-1}, -1, x_{i+1},\dots,x_n) \geq&~ F(x_1,\dots,x_{i-1}, 1, x_{i+1},\dots,x_n)~~~\forall x\in \{\pm 1\}^n.
\end{align*}
\end{definition}

Let $H,\ov{H}$ be the indicator set of two unate functions and $H\subset \ov{H}$. The boundary of $H$ is denoted by $\partial H:=\ov{H}\backslash H$.
\begin{definition}[Semi-thin]
For $\alpha\in [0, 1]$, we say $\partial H$ is $\alpha$-semi thin if for all $x\in H$, at least $\alpha$-fraction of its hypercube neighbors (different in one coordinate) are not in $\partial {H}$. 
\end{definition}

Now, we state the main theorem of this section:
\begin{theorem}[Robust Littlewood-Offord theorem for hyperbolic cones]\label{thm:anticon_good}
Let $\alpha\in [0, 1], \rho >0$. Let $\{x^j_i\}_{i\in [n], j\in [2]}$ be $2n$ vectors in $\R^m$ such that $x_i^1\in \Lambda_{+,h_1}, x_i^2\in (-\Lambda_{+,h_2})$ for all $i\in [n]$.  If there are at least $\alpha$-fraction of $i\in [n]$ such that $\lambda_{\min,h_1}(x_i^1) \geq \rho$ and $\lambda_{\max,h_2}(x_i^2)\leq -\rho$, then we have
\begin{align*}
    \Pr_{\epsilon\sim \{-1,1\}^n}\left[ \exists j\in [2]:~\lambda_{\max,h_j}\left(\sum_{i=1}^n \epsilon_i x_i^j-y_j\right)\in (-2\rho,0] \right]\leq O\left(\frac{1}{\alpha\sqrt{n}}\right).
\end{align*}
\end{theorem}

\begin{proof}
For each $j\in [2]$, define two sets:
\begin{align*}
    H_j := &~ \left\{\epsilon\in \{-1,1\}^n: \lambda_{\max,h_1}\left(\sum_{i=1}^n \epsilon_i x_i^j\right)\leq -2\rho\right\},\\
    \ov{H_j}:= &~ \left\{\epsilon\in \{-1,1\}^n: \lambda_{\max,h_2}\left(\sum_{i=1}^n \epsilon_i x_i^j\right)\leq 0\right\}.
\end{align*}
Then, we have
\begin{align*}
    \partial H_j := \ov{H_j}\backslash H_j = \left\{\epsilon\in \{-1,1\}^n: \lambda_{\max,h_j}\left(\sum_{i=1}^n \epsilon_i x_i^j\right)\in (-2\rho, 0]\right\}.
\end{align*}
Define $F := H_1 \cap H_2$ and $\partial F := (\ov{H_1} \cap \ov{H_2})\backslash F$. Hence,
\begin{align*}
    \partial F = \left\{\epsilon\in \{-1,1\}^n: \exists j\in [2]~\text{s.t.}~ \lambda_{\max,h_j}\left(\sum_{i=1}^n \epsilon_i x_i^j\right)\in (-2\rho, 0]\right\}.
\end{align*}
For  any $\epsilon\in H_1$, consider its hypercube-neighbour $\epsilon'$ which flip the $k$-th coordinate of $\epsilon$. If $\epsilon'\in \partial H_1$, then we have
\begin{align*}
    \lambda_{\max,h_1}\left(\sum_{i=1}^n \epsilon_i x_i^1 - 2\epsilon_k x_k^1\right)\in (-2\rho, 0], \quad \lambda_{\max,h_1}\left(\sum_{i=1}^n \epsilon_i x_i^1\right)\leq -2\rho.
\end{align*}
It implies that $\epsilon_k = -1$. By the fact that $\lambda_{\max}(x+y)\geq \lambda_{\max}(x)+\lambda_{\min}(y)$, we have
\begin{align*}
    \lambda_{\max,h_1}\left(\sum_{i=1}^n \epsilon_i x_i^1\right) + 2\lambda_{\min,h_1}(x_k^1)\leq \lambda_{\max}\left(\sum_{i=1}^n \epsilon_i x_i^1 + 2x_k^1\right)\leq 0,
\end{align*}
which means $\lambda_{\min,h_1}(x_k^1)\leq \rho$. However, we assume that there are $\alpha$-fraction of $k\in [n]$ such that $\lambda_{\min,h_1}(x_k^1)\geq \rho$. Hence, $H_1$ is $\alpha$-semi thin. 

Similarly, for $\epsilon\in H_2$ and its hypercube-neighbor $\epsilon'$ with the $k$-th coordinate flipped, if $\epsilon'\in \partial H_2$, we have
\begin{align*}
    \lambda_{\max,h_2}\left(\sum_{i=1}^n \epsilon_i x_i^2 -2 x_k^2\right)\in (-2\rho, 0], \quad \lambda_{\max,h_2}\left(\sum_{i=1}^n \epsilon_i x_i^2\right)\leq -2\rho.
\end{align*}
Hence,
\begin{align*}
    \lambda_{\max,h_2}\left(\sum_{i=1}^n \epsilon_i x_i^2\right) + 2\lambda_{\min,h_2}(-x_k^2) = \lambda_{\max,h_2}\left(\sum_{i=1}^n \epsilon_i x_i^2\right) - 2\lambda_{\max,h_2}(x_k^2)\leq 0,
\end{align*}
which implies $\lambda_{\max,h_2}(x_k^2)\geq -\rho$. Then, by our assumption, $H_2$ is also $\alpha$-semi thin.

Thus, by Theorem 7.18 in \cite{ost19}, we have
\begin{align*}
    \mathrm{vol}(\partial F)\leq O(1/(\alpha\sqrt{n})),
\end{align*}
which implies the probability upper bound in the lemma.
\end{proof}

In order to satisfy the $\alpha$-semi thin condition in Theorem~\ref{thm:anticon_good}, we use the following lemma using random hash function to bucket the vectors such that the resulting distribution will make the condition hold.

\begin{lemma}[Lemma 46 in \cite{ay21}]\label{lem:bound_good_buckets}
Let $\tau \in (0, \frac{1}{100\sqrt{\log d}}]$. Let $\{x_i\}_{i\in [n]}\subset \R^m$ be a sequence of vectors in the hyperbolic cone $\Lambda_{+}$ of $h$ such that
\begin{align*}
    \lambda_{\max}(x_i)\leq \tau, ~~ \sum_{i=1}^n \lambda_{\min}(x_i)^2 \geq 1~~~\forall i\in [n].
\end{align*}
Let $p\geq \frac{1}{10\tau^2\log d}$ and $\pi:[n]\rightarrow [p]$ be a random hash function that independently assigns each $i\in [n]$ to a uniformly random bucket in $[p]$. For each $c\in [p]$, define $\sigma_c := \sum_{i\in \pi^{-1}(c)}x_i$. And we say $c\in [p]$ is good if $\lambda_{\min}(\sigma_c)\geq \frac{1}{2\tau p}$. 

Then, we have
\begin{align*}
    \Pr\left[|\{c\in [p]: c~\text{is good}\}|\leq \frac{4}{5}p\right]\leq \exp(-p/4).
\end{align*}
\end{lemma}

\begin{proof}
Fix $c\in [p]$. Define indicator random variables $z_i\in \{0,1\}$ for $i\in [n]$ such that $z_i=1$ if $\pi(i)=c$. Since $\pi$ is a random hash function, we have $\Pr[z_i = 1] = \frac{1}{p}$. Then, consider the random vectors $\{z_i x_i\}_{i\in [n]}$. For each $x_i$, $\supp(z_i x_i)\in \Lambda_+$ and $\lambda_{\max}(z_i x_i)\leq \tau$. We note that $\sigma_c = \sum_{i=1}^n z_i x_i$, and
\begin{align*}
    \mu_{\min} = &~ \sum_{i=1}^n \E[\lambda_{\min}(z_ix_i)]= \frac{1}{m}\sum_{i=1}^n \lambda_{\min}(x_i)\\
    \geq &~ \frac{1}{m}\sum_{i=1}^n \lambda_{\min}(x_i)^2\cdot \frac{1}{\lambda_{\max}(x_i)}\\
    \geq &~ \frac{1}{\tau m}.
\end{align*}
Then, by Theorem~\ref{thm:chernoff_cone} with $\delta = 1/2, \mu_{\min}=1/(\tau p), R = \tau$, we have
\begin{align*}
    \Pr\left[\lambda_{\min}\left(\sum_{i=1}^n z_ix_i\right) \leq \frac{1}{2\tau p} \right]\leq d \cdot (2/e)^{\frac{1}{2\tau^2 p}} \leq \frac{1}{10},
\end{align*}
where the last step follows from $p\geq \frac{1}{10\tau^2\log d}$. That is,
\begin{align}\label{eq:bound_sigma_c}
    \Pr\left[\lambda_{\min}(\sigma_c)\geq \frac{1}{2\tau p}\right] \geq 1 - \frac{1}{10} = \frac{9}{10}.
\end{align}

Then, for all $c\in [p]$ and $i\in [n]$, define the indicator variables $B_{c,i}:=\mathbf{1}[\pi(i) = c]$.  Then, $\{B_{c,i}\}_{c\in [p],i\in [n]}$ are negatively associated by a balls and bins argument (see \cite{mr95} for details). Now, the even that $c$ is good can be represented by the indicator variables $G_c:=\mathbf{1}[\lambda_{\min}(\sum_{i=1}^n B_{c,i}x_i)\geq \frac{1}{2\tau p}]$ for $c\in [p]$, which is constructed by applying a monotone non-decreasing function to $\{B_{c,i}\}_{i\in[n]}$. Hence, we know that $\{G_c\}_{c\in [p]}$ are also negatively associated. By Eq.~\eqref{eq:bound_sigma_c}, we have $\E[G_c]\geq \frac{9}{10}$. Thus, by the Chernoff bound for negatively associated random variables, we have
\begin{align*}
    \Pr\left[\sum_{i=1}^p G_c\leq \frac{4}{5}p\right]\leq \exp(-p/4),
\end{align*}
which completes the proof of the lemma.
\end{proof}
\section{Discrepancy result}
In this section, we will show how to relax the isotropic condition in the hyperbolic kadison-Singer theorem~\cite{b18}. And we will apply our hyperbolic concentration result to prove a discrepancy upper bound that works for general vectors.
\subsection{Preliminaries}\label{sec:disp_prelim}
In this section, we formally state some matrix discrepancy results. We first formally state the discrepancy theorem implied by Kadison-Singer theorem.

\begin{theorem}[\cite{mss15}]
Let $x_1, \dots, x_n \in \C^m$ and suppose $\|x_ix_i^*\|\leq \epsilon$ for all $i\in [n]$ and $\sum_{i=1}^n x_ix_i^* = I$. Then, there exist signs $r\in \{-1, 1\}^n$ such that
\begin{align*}
    \left\|\sum_{i=1}^n r_i x_ix_i^*\right\|\leq O(\sqrt{\epsilon}).
\end{align*}
\end{theorem}

This theorem also holds for high rank matrices as long as the isotropic condition holds:
\begin{theorem}[High rank Kadison-Singer \cite{c16,b18}]
Let $X_1, \dots, X_n \in \C^{d\times d}$ be positive semi-definite symmetric matrices such that $\tr[X_i]\leq \epsilon$ for all $i\in [n]$ and $\sum_{i=1}^n X_i = I$. Then, there exist signs $r\in \{-1, 1\}^n$ such that
\begin{align*}
    \left\|\sum_{i=1}^n r_i X_i\right\|\leq  O( \sqrt{ \epsilon } ).
\end{align*}
\end{theorem}

\cite{kls19} showed that the isotropic condition is not necessary for rank-1 matrices:
\begin{theorem}[Rank-1 matrix Spencer, \cite{kls19}]\label{thm:kls19}
Given $n$ vectors $x_1,\dots,x_n\in \C^m$. Let $\sigma^2 = \|\sum_{i=1}^n (x_ix_i^*)^2\|$. Then, there exists a choice of signs $r\in \{-1,1\}^n$ such that 
\begin{align*}
    \left\|\sum_{i=1}^n r_i x_ix_i^*\right\|\leq 4\sigma.
\end{align*}
\end{theorem}

\subsection{Hyperbolic Kadison-Singer with relaxed condition}\label{sec:hy_ks}
The goal of this section is to prove Theorem~\ref{thm:branden_subisotropic}, which relaxes the isotropic condition in Corollary~\ref{cor:branden_discrepancy} to the bounded hyperbolic norm. 

We first formally state the upper bound in \cite{b18}:
\begin{definition}
For $r\in \mathbb{N}_+$, let $U_r$ be the set of all pairs $(\delta, \mu)\in \R_+\times \R_+$ such that
\begin{align*}
    \delta-1\geq \frac{\delta}{\mu}\cdot \frac{\left(1+\frac{\delta}{r\mu}\right)^{r-1}- \left(\frac{\delta}{r\mu}\right)^{r-1}}{\left(1+\frac{\delta}{r\mu}\right)^{r}- \left(\frac{\delta}{r\mu}\right)^{r}},
\end{align*}
and either $\mu>1$, or $\delta\in [1,2],\mu>1-\delta/r$.

Then, the upper bound in \cite{b18} is:
\begin{align*}
    \delta(\epsilon, n, r):=\inf_{(\delta, \mu)\in U_r}~~\frac{\epsilon\mu + (1-\frac{1}{n})\delta}{1+\frac{\mu-1}{n}}.
\end{align*}

In particular, $\delta(\epsilon, \infty, r):= \inf_{(\delta, \mu)\in U_r} \epsilon\mu + \delta$.
\end{definition}

\begin{theorem}[\cite{b18}]\label{thm:bra18_formal}
Let $k \geq 2$ be an integer and $\epsilon$ a positive real number. Suppose $h$ is hyperbolic with respect to $e\in \R^m$, and let $x_1,\dots,x_n \in \Lambda_+(h,e)$ be such that
\begin{align*}
    \tr_h[x_i] \leq \epsilon, ~~\rank(x_i)\leq r~~\forall i\in [n], ~\text{and}\quad \sum_{i=1}^n x_i=e.
\end{align*}
Then there is a partition $S_1 \cup S_2 \cup \cdots \cup S_k = [n]$ such that for all $j\in [k]$,
\begin{align*}
    \left\| \sum_{i\in S_j} x_i\right\|_h \leq \frac{1}{k}\cdot \delta\left(k\epsilon, n, rk\right).
\end{align*}
\end{theorem}
The high-level idea of proving Theorem~\ref{thm:bra18_formal} is similar to \cite{mss15}. We can show that this discrepancy upper-bound can be obtained by rounding a compatible family of polynomials, which is a generalization of the interlacing family defined in \cite{mss15}. Then, this rounding problem is further equivalent to upper-bound the largest root of a mixed hyperbolic polynomial (Definition~\ref{def:mixed_hyperbolic}), which is achieved by proving a structural result about the hyperbolic cone of the mixed hyperbolic polynomial.

Following this approach, we slightly generalize Theorem~\ref{thm:bra18_formal} by relaxing the isotropic condition:
\begin{theorem}\label{thm:branden_subisotropic}
Let $k \geq 2$ be an integer and $\epsilon,\sigma>0$. Suppose $h\in \R[z_1,\dots,z_m]$ is hyperbolic with respect to $e \in \R^m$, and let $x_1, \dots , x_n$ be $n$ vectors in the hyperbolic cone $\Lambda_+(h, e)$ (see Definition~\ref{def:hyper_cone}) such that
\begin{align*}
    \tr_h[x_i] \leq \epsilon, ~~\rank(x_i)\leq r~~\forall i\in [n], ~\text{and}\quad \left\|\sum_{i=1}^n x_i\right\|_h \leq \sigma.
\end{align*}

Then, there exists a partition $S_1 \cup S_2 \cup \cdots \cup S_k = [n]$ such that for all $j\in [k]$,
\begin{align*}
    \left\| \sum_{i\in S_j} x_i\right\|_h \leq \frac{\sigma}{k}\cdot \delta\left(\frac{k\epsilon}{\sigma}, n, rk\right).
\end{align*}
\end{theorem}

\begin{remark}\label{rmk:bound_inf}
By Eq.~(1.7) in \cite{b18}, the above bound is at most 
\begin{align*}
    \frac{\sigma}{k}\cdot \delta(k\epsilon/\sigma, \infty, \infty) = \frac{\sigma}{k}\cdot (1+\sqrt{k\epsilon/\sigma})^2 = \left(\sqrt{\epsilon} + \sqrt{\sigma/k}\right)^2,
\end{align*}
which also generalizes the result of \cite{mss15} (Theorem~\ref{thm:kadison_singer}) to hyperbolic polynomials with sub-isotropic condition.
\end{remark}

\begin{remark}\label{rmk:compare_Branden}
We note that a naive approach to relax the isotropic condition is to add some dummy vectors and then apply Theorem~\ref{thm:bra18_formal}. However, to satisfy the condition that each vector has trace at most $\epsilon$, the number of dummy vector can be $O(n/\epsilon)$ in the worst case. Then, this approach results in an upper bound of $\frac{\sigma}{k}\cdot \delta\left(\frac{k\epsilon}{\sigma},O(n/\epsilon),rk\right)$. By the property of the $\delta$ function, we know that this bound is worse than ours in Theorem~\ref{thm:branden_subisotropic}.  
\end{remark}

The proof of Theorem~\ref{thm:branden_subisotropic} is almost the same as the proof of Theorem 1.3 in \cite{b18}, but relies on the sub-isotropic version of the following theorem. Therefore, we will only prove Theorem~\ref{thm:branden_6.1}.
\begin{theorem}[Sub-isotropic version of Theorem 6.1 in \cite{b18}]\label{thm:branden_6.1}
Suppose $h\in \R[z_1,\dots,z_m]$ is a hyperbolic polynomial with respect to $e\in \R^m$. Let $\mathsf{x}_1,\dots,\mathsf{x}_m$ be independent random vectors in $\Lambda_+(e)$ with finite supports such that
\begin{align*}
   \tr_h[\E[\mathsf{x}_i]] \leq \epsilon, ~~\rank(\E[\mathsf{x}_i])\leq r~~\forall i\in [n], ~\text{and}\quad \left\|\sum_{i=1}^n \E[\mathsf{x}_i]\right\|_h \leq \sigma. 
\end{align*}
Then, we have
\begin{align*}
    \Pr\left[\lambda_{\max}\left(\sum_{i=1}^n \mathsf{x_i}\right)\leq \sigma\cdot \delta(\epsilon/\sigma, n, r)\right]>0.
\end{align*}
\end{theorem}
\begin{proof}
Let $V_i$ be the support of $\mathsf{x}_i$ for $i\in [n]$. By Theorem~\ref{thm:compatible}, the family $\{h[v_1,\dots,v_m](t\ov{e}+\underline{\mathbf{1}})\}_{v_i\in V_i}$ is compatible, where $t\ov{e}+\underline{\mathbf{1}} = \begin{bmatrix}te \\ \mathbf{1}\end{bmatrix}\in \R^{n+m}$

By Theorem~\ref{thm:compatible_root_bound}, there exists $(v_1^*,\dots,v_n^*)\in V_1\times \cdots \times V_n$ with nonzero probability, such that the largest root of $h[v_1^*,\dots,v_n^*](t\ov{e}+\underline{\mathbf{1}})$ is at most the largest root of $\E[h[\mathsf{x}_1,\dots,\mathsf{x}_n]]$.

By Fact~\ref{fac:affine_linear}, $\E[h[\mathsf{x}_1,\dots,\mathsf{x}_n]] = h[\E[\mathsf{x}_1],\dots,\E[\mathsf{x}_n]]$. Let $\lambda_{\max}(v_1,\dots,v_n)$ denote the largest root of $h[v_1,\dots,v_m](t\ov{e}+\underline{\mathbf{1}})$. Then, we have
\begin{align*}
    \lambda_{\max}(\E[\mathsf{x}_1],\dots,\E[\mathsf{x}_n])\geq ~ \lambda_{\max}(v_1^*,\dots,v_n^*)\geq~ \lambda_{\max}(v_1^*+\cdots +v_n^*),
\end{align*}
where the second step follows from Theorem~\ref{thm:sum_root_bound_mix}.

It is easy to verify that $\E[\mathsf{x}_1],\dots, \E[\mathsf{x}_n]$ satisfy the conditions in Theorem~\ref{thm:mix_hyper_root}. Thus, by Theorem~\ref{thm:mix_hyper_root}, we get that
\begin{align*}
    \lambda_{\max}(v_1^*+\cdots +v_n^*) \leq ~ \lambda_{\max}(\E[\mathsf{x}_1],\dots,\E[\mathsf{x}_n]) \leq ~ \sigma\cdot \delta(\epsilon/\sigma, n, r),
\end{align*}
which completes the proof.
\end{proof}

Similar to Corollary~\ref{cor:branden_discrepancy}, Theorem~\ref{thm:branden_subisotropic} also implies the following discrepancy result for vectors in sub-isotropic position.
\begin{corollary}
Let $0<\epsilon\leq \frac{1}{2}$. Suppose $h\in \R[z_1,\dots,z_m]$ is hyperbolic with respect to $e\in \R^m$, and let $x_1, \dots , x_n\in \Lambda_+(h,e)$ that satisfy 
\begin{align*}
    \tr_h[x_i] \leq \epsilon, ~\text{and}\quad \left\|\sum_{i=1}^n x_i\right\|_h \leq \sigma. 
\end{align*}
Then, there exist signs $r\in \{-1,1\}^n$ such that
\begin{align*}
    \left\|\sum_{i=1}^n r_i x_i\right\|_h\leq 2\sqrt{\epsilon(2\sigma-\epsilon)}.
\end{align*}
\end{corollary}
\begin{proof}
By Theorem~\ref{thm:branden_subisotropic} with $k=2$ and the upper bound in Remark~\ref{rmk:bound_inf}, there exists a set $S\subseteq [n]$ such that
\begin{align*}
    \Big\|\sum_{i\in S} x_i\Big\|_h \leq (\sqrt{\epsilon} + \sqrt{\sigma/2})^2, ~~\text{and}~~\Big\|\sum_{i\not\in S} x_i\Big\|_h \leq (\sqrt{\epsilon} + \sqrt{\sigma/2})^2.
\end{align*}
Since we know that $\|\sum_{i=1}^n x_i\|_h \leq \sigma$, we get that
\begin{align*}
    \Big\|\sum_{i\in S} - \sum_{i\not\in S} x_i\Big\|_h\leq \sigma - 2(\sqrt{\epsilon} + \sqrt{\sigma/2})^2 = 2\sqrt{\epsilon(2\sigma-\epsilon)}.
\end{align*}
By assigning $r_i=1$ for $i\in S$ and $r_i=-1$ for $i\notin S$, we complete the proof of the corollary.
\end{proof}

\subsubsection{Technical tools in previous work}
In this section, we provide some necessary definitions and technical tools we used in \cite{b18}.
\begin{definition}[Directional derivative]
Let $h\in \R[x_1,\dots,x_m]$. The directional derivative of $h(x)$ with respect to $v\in \R^m$ is defined as
\begin{align*}
    D_vh(x) := \sum_{i=1}^m v_i \cdot \frac{\partial h}{\partial x_i} (x).
\end{align*}
\end{definition}

The following fact shows the relation between directional derivative and the usual derivative.
\begin{fact}\label{fac:dir_derivate}
For any polynomial $h(x)$ and any vector $v\in \R^m$, we have
\begin{align*}
    D_vh(x+tv)= \frac{\d}{\d t} h(x + tv).
\end{align*}
\end{fact}
If $h$ is a hyperbolic polynomial, then the directional derivative is related to the hyperbolic trace:
\begin{fact}\label{fac:d_trace}
If $h$ is hyperbolic with respect to $e\in \R^m$, then for any $v\in \R^m$, we have
\begin{align*}
    \tr_h[v] = \frac{D_v h(e)}{h(e)}.
\end{align*}
\end{fact}

\begin{definition}[Mixed hyperbolic polynomial]\label{def:mixed_hyperbolic}
If $h(x)\in \R[x_1,\dots,x_m]$ is a hyperbolic polynomial with respect to $e\in \R^m$, and $v_1,\dots,v_n\in \Lambda_+$, then the mixed hyperbolic polynomial $h[v_1,\dots,v_m]\in \R[x_1,\dots,x_m,y_1,\dots,y_n]$ is defined as
\begin{align*}
    h[v_1,\dots,v_n]:=\prod_{i=1}^m (1-y_i D_{v_i}) h(x).
\end{align*}
\end{definition}

Br{\"a}nd{\'e}n \cite{b18} proved that $h[v_1,\dots,v_n]$ is also hyperbolic with the hyperbolic cone containing $\Lambda_{++}\times \R^n_{\leq 0}$. In our proof, we will also use the following fact, which can be easily proved by showing that $h[v_1,\dots,v_n]$ is affine linear in each coordinate. 

\begin{fact}\label{fac:affine_linear}
Let $\mathsf{x}_1,\dots,\mathsf{x}_n$ be independent random variables in $\R^m$. Then,
\begin{align*}
    \E[h[\mathsf{x}_1,\dots,\mathsf{x}_n]] = h[\E[\mathsf{x}_1],\dots,\E[\mathsf{x}_n]].
\end{align*}
\end{fact}

Br{\"a}nd{\'e}n \cite{b18} also defined the compatible family of polynomials, which is a sub-class of interlacing family of polynomials in \cite{mss15,mss18}.
\begin{definition}[Compatible family of polynomials]\label{def:compatible}
Let $S_1,\dots,S_n$ be finite sets. A family of polynomials 
\begin{align*}
    {\cal F}=\{f(S;t)\}_{S\in S_1\times \cdots \times S_n}\subset \R[t]
\end{align*}
is called compatible if the following properties hold:
\begin{itemize}
    \item all the nonzero members of F have the same degree and the same signs of their leading coefficients, and
    \item for all choices of independent random variables $\mathsf{x}_1\in S_1,\dots,\mathsf{x}_n\in S_n$, the polynomial
    \begin{align*}
        \E[f(\mathsf{x}_1,\dots,\mathsf{x}_n; t)]
    \end{align*}
    is real-rooted.
\end{itemize}
\end{definition}

The following theorem characterizes the largest root of the expectation polynomial in the compatible family, which is very similar to the result for interlacing family \cite{mss15}.
\begin{theorem}[Theorem 2.3 in \cite{b18}]\label{thm:compatible_root_bound}
Let $\{f(S;t)\}_{S\in S_1\times \cdots \times S_n}$ be a compatible family, and let $\mathsf{x}_1\in S_1,\dots,\mathsf{x}_n\in S_n$ be independent random variables such that $\E[f(\mathsf{x}_1,\dots,\mathsf{x}_n)]\not \equiv 0$. 

Then there is a tuple $S = (s_1,\dots, s_n) \in S_1\times \cdots \times S_n$, with $\Pr[\mathsf{x}_i = s_i] > 0$ for all $i\in [n]$, such that the largest root of $f(s_1,\cdots, s_n;t)$ is smaller or equal to the largest root of $\E[f(\mathsf{x}_1,\cdots,\mathsf{x}_n;t)]$.

\end{theorem}

The theorem below shows that mixed hyperbolic polynomials form a compatible family.
\begin{theorem}[Theorem 3.5 in \cite{b18}]\label{thm:compatible}
Let $h(x)$ be hyperbolic with respect to $e\in \R^m$, and let $V_1,\dots,V_n$ be finite sets of vectors in $\Lambda_+$. Let $w\in \R^{m+n}$. For $V=(v_1,\dots,v_n)\in V_1\times \cdots \times V_n$, define
\begin{align*}
    f(V;t):=h[v_1,\dots,v_n](t\ov{e}+w),
\end{align*}
where $\ov{e}:=\begin{bmatrix}e \\ 0\end{bmatrix}\in \R^{n+m}$.
Then, $\{f(V;t)\}_{V\in V_1\times \cdots \times V_n}$ is a compatible family.
\end{theorem}

Let $\lambda_{\max}(v_1,\dots,v_n)$ denote the largest root of the mixed hyperbolic polynomial $h[v_1,\dots,v_m](t\ov{e} + \underline{\mathbf{1}})\in \R[t]$, i.e.,
\begin{align}\label{eq:def_lambda_mix}
    \lambda_{\max}(v_1,\dots,v_n):=\lambda_{\max}(h[v_1,\dots,v_m](t\ov{e} + \underline{\mathbf{1}}))
\end{align}
The following theorem shows that $\lambda_{\max}(v_1,\dots,v_n)$ can upper-bounds the largest hyperbolic eigenvalue of the vector $v_1+\cdots +v_n$.

\begin{theorem}[Theorem 5.2 in \cite{b18}]\label{thm:sum_root_bound_mix}
If $h$ is hyperbolic with respect to $e$ and $v_1,\dots,v_n\in \Lambda_+(e)$, then
\begin{align*}
    \lambda_{\max}(v_1+\cdots+v_n)\leq \lambda_{\max}(v_1,\dots,v_n).
\end{align*}
\end{theorem}
The following theorem shows a connection between the hyperbolic cone of $h$ and the hyperbolic cone of the mixed hyperbolic polynomial $h[v_1,\dots,v_n]$.
\begin{theorem}[Corollary 5.5 in \cite{b18}]\label{cor:branden_5.5}
Suppose $h$ is hyperbolic with respect to $e\in \R^m$, and let $\Gamma_+$ be the hyperbolic cone of $h[v_1,\dots,v_n]$, where $v_i\in \Lambda_+(e)$ and $1\leq \rank(v_i)\leq r_i$ for $i\in [m]$. Suppose $x\in \Lambda_{++}(e)$ be such that for $i\in [m]$, $\ov{x}+\mu_i\underline{e_i}\in \Gamma_+$ for any $\mu_i>0$. 

Then, for any $(\delta_i, \mu_i)\in U_{r_i}$ for $i\in [m]$,
\begin{align*}
    \ov{x}+\left(1-\frac{1}{m}\right)\sum_{i=1}^n\delta_i \ov{v_i} + \left(1-\frac{1}{m}\right)\underline{\mathbf{1}} + \frac{1}{m}\sum_{i=1}^n \mu_i \underline{e_i} \in \Gamma_+.
\end{align*}
\end{theorem}

\subsubsection{Upper bound for the largest root of the mixed hyperbolic polynomial}
The goal of this section is to prove Theorem~\ref{thm:mix_hyper_root}, which gives an upper bound for the mixed hyperbolic polynomial with vectors in sub-isotropic position. 
\begin{theorem}[Sub-isotropic version of Theorem 5.6 in \cite{b18}]\label{thm:mix_hyper_root}
Suppose $h\in \R[z_1,\dots,z_m]$ is hyperbolic with respect to $e\in \R^m$, and let $v_1, \dots , v_n\in \Lambda_+(h,e)$ that satisfy 
\begin{align*}
    \tr_h[v_i] \leq \epsilon, ~~\rank(v_i)\leq r~~\forall i\in [n], ~\text{and}\quad \left\|\sum_{i=1}^n v_i\right\|_h \leq \sigma. 
\end{align*}
Then,
\begin{align*}
    \lambda_{\max}(v_1,\dots,v_n)\leq \sigma \cdot \delta(\epsilon/\sigma, n, r),
\end{align*}
where $\lambda_{\max}(v_1,\dots,v_n)$ is defined in Eq.~\eqref{eq:def_lambda_mix}.
\end{theorem}

\begin{proof}
For $\mu>0$, let $x:=\epsilon \mu \cdot e$ and $\mu_i := \mu$ for $i\in [n]$. Let $e_i\in \R^{n}$ be the $i$-th standard basis vector.

Then, we have
\begin{align*}
    h[v_1,\dots,v_n](\ov{x}+\mu_i\underline{e_i}) = &~ (1-\mu D_{v_i})h(\epsilon\mu e)\\
    = &~ \epsilon_d \mu^d h(e) + \mu^d\epsilon^{d-1}D_{v_i} h(e)\\
    = &~ \mu^d\epsilon^{d-1}h(e)(\epsilon - \tr_h[v_i])\\
    > &~ 0,
\end{align*}
where the first step follows from Fact~\ref{fac:dir_derivate}, the second step follows from the homogeneity of hyperbolic polynomials, and the third step follows from Fact~\ref{fac:d_trace}. 

By part (2) of Theorem~\ref{thm:hyperbolic_cone_gar}, we get that $x+\mu_i \underline{e_i}\in \Gamma_+$, the hyperbolic cone of $h[v_1,\dots,v_n]$, for all $i\in [n]$.

Then, by Theorem~\ref{cor:branden_5.5}, for any $(\delta, \mu)\in U_r$,
\begin{align*}
    \epsilon \mu \ov{e} + \left(1-\frac{1}{n}\right)\delta \sum_{i=1}^n v_i + (1+\frac{\mu-1}{n})\underline{\mathbf{1}}\in \Gamma_+,
\end{align*}
which implies 
\begin{align*}
    \frac{\epsilon \mu \ov{e} + \left(1-\frac{1}{n}\right)\delta \sum_{i=1}^n v_i }{1+\frac{\mu-1}{n}}+\underline{\mathbf{1}}\in \Gamma_+,
\end{align*}
by the homogeneity of $\Gamma_+$. Since $\ov{e}\in \Gamma_{++}$, $\lambda_{\max}(\sum_{i=1}^n v_i)\leq \sigma$, and $\Gamma_+$ is a convex cone, we have
\begin{align*}
    \frac{\left(\epsilon \mu + \left(1-\frac{1}{n}\right)\delta\sigma\right) }{1+\frac{\mu-1}{n}}\ov{e} +\underline{\mathbf{1}}\in \Gamma_+.
\end{align*}
Hence, by Remark 5.1 in \cite{b18}, 
\begin{align*}
    \lambda_{\max}(v_1,\dots,v_n) = \inf_{\rho>0}~~ \rho \ov{e} + \underline{\mathbf{1}}\in \Gamma_+.
\end{align*}
Hence, we conclude that
\begin{align*}
    \lambda_{\max}(v_1,\dots,v_n)\leq &~ \inf_{(\delta, \mu)\in U_r}~~\frac{\left(\epsilon \mu + \left(1-\frac{1}{n}\right)\delta\sigma\right) }{1+\frac{\mu-1}{n}}\\
    = &~ \sigma \cdot \delta(\epsilon/\sigma, n, r).
\end{align*}
\end{proof}

\subsection{Discrepancy result with high probability}

The goal of this section is to prove Theorem~\ref{thm:eight_deviations}, which proves the rank-1 case of the hyperbolic Spencer conjecture (Conjecture~\ref{conj:hyperbolic_spencer}).

\begin{theorem}[Eight deviations suffice]\label{thm:eight_deviations}
Given $x_1, x_2, \cdots, x_n \in \R^m$ such that $\rank(x_i)\leq 1$ for all $i\in [n]$. Let $h$ be an $m$-variable, degree-$d$ hyperbolic polynomial with respect to $e$. Let $\sigma = ( \sum_{i=1}^n \|x_i\|_h^2 )^{1/2}$. Then, there exists a sign vector $r \sim \{-1, 1\}^n$ such that
\begin{align*}
    \left\| \sum_{i=1}^n r_i x_i \right\|_h \leq  8\sigma 
\end{align*}
holds.
\end{theorem}
\begin{proof}
Similar to the proof of Theorem~\ref{thm:main}, we first have
\begin{align*}
    \E_{ r \sim \{ \pm 1 \}^n }\left[\left\|\sum_{i=1}^n r_i x_i\right\|_h\right]
    \leq & ~ \left( \E_{ r \sim \{ \pm 1 \}^n } \left[\left\|\sum_{i=1}^n r_i x_i\right\|_{h,2q}^{2q}\right] \right)^{1/(2q)}\\
    \leq & ~ \sqrt{2q-1} \cdot \left( \sum_{i=1}^n\|x_i\|_{h}^2 \right)^{1/2}\\
    = &~ \sqrt{2q-1} \cdot \sigma,
\end{align*}
where the first step follows from Eq.~\eqref{eq:bound_expect_h_norm}. 

By setting $q=1$, we have
\begin{align*}
    \E_{ r \sim \{ \pm 1 \}^n }\left[\left\|\sum_{i=1}^n r_i x_i\right\|_h\right]
    \leq  \sigma.
\end{align*}

By Claim~\ref{clm:hyperbolic_2_1_norm}, 
\begin{align*}
    \E_{ r \sim \{ \pm 1 \}^n }\left[\left\|\sum_{i=1}^n r_i x_i\right\|_h^2\right] 
    \leq 2\left( \E_{ r \sim \{ \pm 1 \}^n } \left[\left\|\sum_{i=1}^n r_i x_i\right\|_h\right] \right)^2 \leq 2\sigma^2.
\end{align*}
Then, by Corollary~\ref{cor:hyperbolic_prob},
\begin{align*}
    \Pr_{ r \sim \{ \pm 1 \}^n } \left[ \left\| \sum_{i=1}^n r_i x_i\right\|_h > t \right] \leq &~ 2 \exp \left(-\frac{t^2}{32 \E_{ r \sim \{ \pm 1 \}^n }[ \| \sum_{i=1}^n r_i x_i \|_h^2 ] }\right)\\
    \leq &~ 2\exp\left(-\frac{t^2}{64 \sigma^2}\right).
\end{align*}
By choosing $t=8\sigma$, we have
\begin{align*}
    \Pr_{ r \sim \{ \pm 1 \}^n } \left[ \left\| \sum_{i=1}^n r_i x_i\right\|_h > 8\sigma \right]\leq 2/e.
\end{align*}
Therefore, with probability $1-2/e$, we have
\begin{align*}
    \left\| \sum_{i=1}^n r_i x_i\right\|_h \leq 8\sigma,
\end{align*}
which proves the theorem.
\end{proof}

\begin{remark}
It is interesting to apply Theorem~\ref{thm:eight_deviations} to determinant polynomial $h(x)=\det(X)$. It implies that for rank-1 matrices $X_1,\dots,X_n\in \R^{d\times d}$, 
\begin{align*}
    \Pr_{r\sim \{\pm 1\}^n}\left[\left\|\sum_{i=1}^n r_i X_i\right\|>t\right]\leq 2\exp\left(-\frac{t^2}{64 \sigma^2}\right),
\end{align*}
for $\sigma^2=\sum_{i=1}^n \|X_i\|^2$.

This result is in fact incomparable to the matrix Chernoff bound \cite{tro15}, which shows that
\begin{align*}
        \Pr_{r\sim \{\pm 1\}^n}\left[\left\|\sum_{i=1}^n r_i X_i\right\|>t\right]\leq 2d\cdot \exp\left(-\frac{t^2}{2\wt{\sigma}^2}\right),
\end{align*}
where $\wt{\sigma}^2=\|\sum_{i=1}^n X_i^2\|$. Because we only know the following relation between $\sigma$ and $\wt{\sigma}$ \cite{tro15}:
\begin{align*}
    \wt{\sigma}^2 \leq \sigma^2 \leq d\cdot \wt{\sigma}^2.
\end{align*}
\end{remark}

\else
\appendix
\input{basic_hyperbolic}
\fi
\ifdefined\isarxiv
\section*{Acknowledgements}
We thank the anonymous reviewers for helpful comments.
The authors would like to thank Petter Br{\"a}nd{\'e}n and James Renegar  for many useful discussions about the literature of hyperbolic polynomials.The authors would like to thank Yin Tat Lee and James Renegar, Scott Aaronson  for encouraging us to work on this topic. The authors would like to thank Dana Moshkovitz for giving comments on the draft.

Ruizhe Zhang was supported by NSF Grant CCF-1648712.

\addcontentsline{toc}{section}{References}
\fi
\ifdefined\isarxiv
\bibliographystyle{alpha}
\bibliography{ref}
\fi

\end{document}